\documentclass{amsart}
\pdfoutput=1
\usepackage[headings]{fullpage}
\usepackage{amssymb,epic,eepic,epsfig,amsbsy,amsmath,amscd,color,amsthm}
\usepackage[all,cmtip]{xy}
\usepackage{amssymb,amsmath,amscd,color}
\usepackage{graphicx}
\usepackage{texdraw}
\usepackage{url}
\usepackage{mathrsfs}

\usepackage[utf8]{inputenc}
\usepackage{fontenc}
\usepackage{amsfonts}

\newtheorem{theorem}{Theorem}[section]

\newtheorem{thm}{Theorem}
\newtheorem{lemma}[theorem]{Lemma}

\newtheorem{prop}[theorem]{Proposition}

\newtheorem{defn}[theorem]{Definition}

\theoremstyle{definition}

\newtheorem{rmk}[theorem]{Remark}
\newtheorem{example}[theorem]{Example}
\newtheorem*{Not}{Notations}

\theoremstyle{definition}

\usepackage{tikz}
\tikzstyle cross=[preaction={draw=white, -, line width=6pt}]
\tikzstyle normal=[thick]
\usepackage{braids}
\usetikzlibrary{arrows,decorations.markings}
\usepackage{relsize,exscale}
\usetikzlibrary{cd}

\usetikzlibrary{patterns,arrows,decorations.pathreplacing}

\newcommand{\BC}{\mathbb{C}}
\newcommand{\BN}{\mathbb{N}}
\newcommand{\BQ}{\mathbb{Q}}

\newcommand{\BZ}{\mathbb{Z}}

\newcommand{\CB}{\mathcal{B}}
\newcommand{\CH}{\mathcal{H}}
\def\CL{\mathcal L}
\def\CN{\mathcal N}

\newcommand{\CR}{\mathcal{R}}
\def\CS{\mathcal S}
\def\CU{\mathcal U}

\newcommand{\mt}{\operatorname{\mathtt{t}}}

\newcommand{\Bn}{\mathcal{B}_n}
\newcommand{\PBn}{\mathcal{PB}_n}

\newcommand{\Sk}{\mathfrak{S}}
\newcommand{\perm}{\operatorname{perm}}

\newcommand{\Homeo}{\operatorname{Homeo}}
\newcommand{\Mod}{\operatorname{Mod}}

\newcommand{\bapp}{\left. \begin{array}{rcl}}
\newcommand{\eapp}{\end{array} \right.}
\newcommand{\bfct}{\left\lbrace \begin{array}{rcl}}
\newcommand{\efct}{\end{array} \right.}

\newcommand{\Conf}{\operatorname{Conf}}

\newcommand{\Hlf}{\operatorname{H} ^{\mathrm{lf}}}
\newcommand{\Hnot}{\operatorname{H}}

\newcommand{\Hrelm}{\operatorname{\CH}^{\mathrm{rel }-}}
\newcommand{\Hrelp}{\operatorname{\CH}^{\mathrm{rel }+}}

\newcommand{\Crelm}{\operatorname{C}^{\mathrm{rel }-}}

\newcommand{\Rhom}{\operatorname{R}^{\mathrm{hom}}}

\newcommand{\Laurent}{\CR}

\newcommand{\aug}{\operatorname{aug}}
\def\ab{\mathrm{ab}}

\newcommand{\slt}{{\mathfrak{sl}(2)}}
\newcommand{\Uq}{{U_q\slt}}
\newcommand{\UqhL}{{U^{\frac{L}{2}}_q\slt}}

\newcommand{\qbin}[2]{\left[\begin{array}{c}
      #1 \\
      #2 \end{array}\right]}

\newcommand{\RR}{\operatorname{R}}

\def\bfVl{{\bf V}^l}
\newcommand{\Jones}{\operatorname{J}}

\newcommand{\End}{\operatorname{End}}
\def\Tr{\operatorname{Tr}}
\def\Id{\operatorname{Id}}
\def\Coker{\operatorname{Coker}}

\title{Colored Jones polynomials and abelianized Lefschetz numbers}
\author{Jules Martel}
\date{}

\begin{document}

\maketitle

\begin{abstract}
We show that colored Jones polynomials of the closure of a braid compute weighted sums of abelianized Lefschetz numbers associated with the action of the braid on configuration spaces. The sum is over number of configuration points. Then we interpret this sum in terms of Poincaré--Lefschetz duality intersection pairing between homology classes. 
\end{abstract}

\tableofcontents

\section{Introduction}

\subsection{Colored Jones polynomials and Lefschetz numbers}

The Jones polynomial \cite{Jo} is an invariant of knots that can distinguish a knot from its mirror image, contrary to e.g. the Alexander polynomial. One way of defining this invariant for a braid closure is by taking a (quantum) trace of the braid representation over tensor products of the standard two dimensional module on the quantized algebra associated with $\slt$, denoted $\Uq$. It turns out that this construction generalizes to arbitrary finite dimensional representation of $\Uq$. It gives rise to a graded family of polynomials knot invariants called {\em colored Jones polynomials}. See \cite{Kho} for the definition of these invariants in this framework of $\Uq$ simple finite dimensional representations. One may also find in \cite{Kho} the relation between any colored Jones polynomial and the first of the family, the so called Jones polynomial, that satisfies some local {\em skein relations}. While the Alexander polynomial has several topological interpretations, there is a lack of topological definition of the Jones polynomial in that sense. An example of expectation is the {\em Volume conjecture} (\cite{Ka}), stated in \cite{MuMu} for colored Jones polynomials, suggesting that the family of colored Jones polynomials for a knot contains the simplicial volume of its complement. As colored Jones polynomials are invariants defined in a purely algebraic fashion (using extensively inherent tools of the $\Uq$-modules category), understanding their topological content is a natural question. In the present paper we give some topological interpretation of the Jones polynomials, using some homological model of quantum representations built from {\em configuration spaces of points inside the punctured disks}.

In \cite{Jules1}, the author extends T. Kohno's theorem (\cite{Koh}) relating representations of braid groups over tensor product of some $\Uq$-modules called {\em Verma modules}, and those constructed by R. Lawrence (\cite{Law}) on homologies of configuration spaces of points in the punctured disks. Kohno's theorem has allowed topological interpretation for (colored) Jones polynomials involving homology classes of the configuration space of points in the punctured disk, see \cite{Big3,Ito2,An}. By use of \cite{Jules1} which recovers more quantum structure from Lawrence homological representations, we prove the following result. 

\begin{thm}[Theorem~\ref{JonesisLefschetz}]\label{maintheoremIntro}
Let $\beta$ be a braid on $n$ strands, such that its closure is a knot $K$. Let $\Jones_K(N)$ denote the $N$-colored Jones polynomial of $K$. Then, for $l\in \BN$:
\begin{align*}
\Jones_K(l+1) = q^{-w(\beta)-nl} \sum_{r=0}^{nl} (-1)^r   \CL_H \left( \widehat{\beta}^r \right)  q^{2r} .
%\\
%& = \sum_{r=0}^{(n-1)l} (-1)^r \left( q^{-nl} \left[ \CL_H \left( \widehat{\beta}^r \right) \right]_{\alpha_i=l}  - q^{-n(-l-2)} \left[ \CL_H \left( \widehat{\beta}^r \right) \right]_{\alpha_i=-l-2} \right) q^{2r} \\
%& + q^{-nl} \sum_{r=(n-1)l+1}^{nl} (-1)^r \left[ \CL_H \left( \widehat{\beta}^r \right) \right]_{\alpha_i=l} q^{2r} . %(q^{w(\beta)\frac{(l+1)^2-1}{2}}) 
\end{align*}
where $w$ is the writhe function, and numbers $\CL_H \left( \widehat{\beta}^r \right) $ designate the {\em abelianized Lefschetz numbers} associated with some homeomorphisms of the configuration space of $r$ points inside the punctured disk (more precisely defined in Section \ref{subsectionofmaintheorem}). These homeomorphisms are defined from the braid $\beta$ considered as an isotopy class of the punctured disk. 
%\end{thm}
\end{thm}

The generalized theory of Lefschetz numbers is recalled in Section \ref{surveyLefschetz}. For a homeomorphism $f$, the abelianized Lefschetz number of $f$ counts classes of fixed point of $f$, namely {\em (homological) Nielsen classes} of fixed points. It can be computed from the homological action of $f$ lifted to some (abelian) cover, and it is invariant under isotopy of $f$. For example, the abelianized Lefschetz number of $f$ contains the classical Lefschetz number of $f$ defined to be the sum of $f$ fixed points' degrees.

As abelianized Lefschetz numbers are computable from homology, we interpret  Theorem \ref{maintheoremIntro} using Poincaré--Lefschetz duality, as a weighted sum of intersection pairings between homology classes. It is the content of Theorem \ref{pairingformulaforJones}. One can then obtain the colored Jones polynomials of a knot from intersection theory between homology classes of the configuration space of points in the punctured disk. We perform such a computation for the trefoil knot as an example. 

The Poincaré--Lefschetz duality is applied in the context of homology with local coefficients in some ring of Laurent polynomials in the variable $q$ denoted  $\Laurent$. This ring suits with evaluation of $q$ to roots of unity for instance, which is the context of the volume conjecture for colored Jones polynomials as stated in \cite{MuMu}. More precisely, the formulae from Theorems \ref{maintheoremIntro} and \ref{pairingformulaforJones} still have a topological meaning when variable $q$ is specialized to a root of unity. This is a consequence of the fact that in \cite{Jules1}, Kohno's theorem is extended to the ring $\Laurent$, avoiding generic conditions on parameters. 

\subsection{Plan of the paper}

In Section \ref{Jules1summary}, we present quantum representation of braid groups on tensor products of $\Uq$ Verma modules on one hand. On the other hand, we introduce configuration spaces of points inside punctured disks, their associated homologies with local coefficients denoted $\Hrelm$, and the Lawrence representation of braid groups arising upon modules $\Hrelm$. We then summarize results from \cite{Jules1} relating both representations of braid groups over Verma modules and those over homological modules $\Hrelm$. 

In Section \ref{Sectionofmaintheorem}, we first suggest a general survey on Lefschetz numbers theory and relations with Nielsen fixed points theory. Then we define colored Jones polynomials from Verma modules and we interpret this definition by mean of homology modules $\Hrelm$. Finally we prove Theorem \ref{maintheoremIntro} stating that colored Jones polynomials compute Lefschetz numbers arising from homologies $\Hrelm$. 

In Section \ref{SectionIntersection}, we give a dual interpretation for homological modules $\Hrelm$ by application of the Poincaré--Lefschetz duality. We define homological modules denoted $\Hrelp$ that are dual modules of $\Hrelm$ regarding an intersection pairing. We present dual bases for $\Hrelm$ and $\Hrelp$. It allows to prove Theorem \ref{pairingformulaforJones} providing a formula for colored Jones polynomials in terms of intersection pairings between homology classes living in $\Hrelm \times \Hrelp$. We finally do an example of computation in the case of the trefoil knot. \\

{\bf Acknowledgment} Part of this work was achieved during the PhD of the author. The author thanks very much his advisor F. Costantino for asking this problem, for all his comments and remarks leading to this paper. The author is also very grateful to L.-H. Robert and S. Willets for useful and interesting discussions. 

\section{Homological model for quantum representations}\label{Jules1summary}

This section is devoted to an outline of \cite{Jules1} providing a homological model for some quantum representations. Namely, it leads to an isomorphism of representations of braid groups between homology modules built from configuration spaces of punctured disks and tensor product of (integral) Verma modules of the quantized algebra of $\slt$ denoted $\Uq$. In Section \ref{qVermaRep} we fix the set-up of Verma modules representations, giving all definitions it requires. In Section \ref{HXrmodel} we present the homological construction of representations due to Lawrence, involving definitions of configuration spaces on the punctured disks. In Section \ref{SummaryJules1} we state the theorem relating both representations. We begin with braid group representations because in Section \ref{cJonesSetup} the colored Jones polynomials for knots will be defined from them.

%We begin with recalling the Artin presentation for braid groups.

\begin{defn}[Braid groups]\label{Artinpres}
Let $n\in \BN$. The {\em braid group} on $n$ strands $\Bn$ is the group generated by $n-1$ elements satisfying the so called {\em ``braid relations"}:
$$\Bn = \left\langle \sigma_1,\ldots,\sigma_{n-1} \Big| \begin{array}{ll} \sigma_i \sigma_j = \sigma_j \sigma_i & \text{ if } |i-j| \ge 2 \\ 
\sigma_i \sigma_{i+1} \sigma_i = \sigma_{i+1} \sigma_i \sigma_{i+1} & \text{ for } i=1,\ldots, n-2 \end{array} \right\rangle$$
\end{defn} 

\subsection{Quantum Verma modules}\label{qVermaRep}

We define notations for quantum numbers, factorials and binomials.

\begin{Not}\label{quantumq}
Let $i$ be a positive integer. We define the following elements of $\BZ \left[ q^{\pm 1} \right]$.
\begin{equation*}
\left[ i \right]_q := \frac{q^i-q^{-i}}{q-q^{-1}} , \text{  } \left[ k \right]_q! := \prod_{i=1}^k \left[ i \right]_q , \text{  } \qbin{k}{l}_q := \frac{\left[ k \right]_q!}{\left[ k-l \right]_q! \left[ l \right]_q!} .
\end{equation*}
\end{Not}

\subsubsection{An integral version for $\Uq$}\label{halfLusztigversion}

In this section, we define an integral version for the quantized algebra associated with $\slt$ that will be central for the present work. 
First, we give the most standard definition of the quantum algebra $\Uq$ as a vector space over a rational field.

\begin{defn}\label{Uqnaif}
The algebra $\Uq$ is the algebra over $\BQ(q)$ generated by elements $E,F$ and $K^{\pm 1}$, satisfying the following relations:
\begin{align*}
KEK^{-1}=q^2E & \text{ , } KFK^{-1}=q^{-2}F \\
\left[E, F \right] = \frac{K-K^{-1}}{q-q^{-1}} & \text{ and }
KK^{-1}=K^{-1}K=1 .
\end{align*}
The algebra $\Uq$ is endowed with a coalgebra structure defined by $\Delta$ and $\epsilon$ as follows:
\[
\begin{array}{rl}
\Delta(E)= 1\otimes E+ E\otimes K, & \Delta(F)= K^{-1}\otimes F+ F\otimes 1 \\
\Delta(K) = K \otimes K, & \Delta(K^{-1}) = K^{-1}\otimes K^{-1} \\
\epsilon(E) = \epsilon(F) = 0, & \epsilon(K) = \epsilon(K^{-1}) = 1
\end{array}
\]
and an antipode defined as follows:
\[
S(E) = EK^{-1}, S(F)=-KF,S(K)=K^{-1},S(K^{-1}) = K.
\]
This provides a {\em Hopf algebra} structure, so that the category of modules over $\Uq$ is monoidal.
\end{defn}

We are interested in an {\em integral version} of this algebra (see \cite[Definition~5.3]{Jules1}), namely one over the ring $\Laurent_0 := \BZ\left[ q^{\pm 1} \right]$. Indeed, we define a version similar to the one introduced by Lusztig in \cite{Lus}, but we only consider the so called {\em divided powers of $F$} as generators, not those of $E$. This version is introduced in \cite{Hab,JK,Jules1} (with subtle differences in the definitions of divided powers for $F$). Let:
\[
F^{(n)} :=  \frac{(q-q^{-1})^n}{\left[ n \right]_q!} F^n .
\]
%Let  be the ring of integral Laurent polynomials in the variable $q$. 

\begin{defn}[Half integral algebra]\label{Halflusztig}
Let $\UqhL$ be the $\Laurent_0$-subalgebra of $\Uq$ generated by $E$, $K^{\pm 1}$ and $F^{(n)}$ for $n\in \BN^*$.  
\end{defn}
We call it a {\em half integral version} for $\Uq$, the word half to illustrate that we consider only half of divided powers as generators.

\begin{rmk}[Relations in $\UqhL$, {\cite[(16)~and~(17)]{JK}}]\label{relationsUqhL}
The relations among generators involving divided powers are the following ones:
\[
KF^{(n)}K^{-1} = q^{-2n}F^{(n)}
\]
\[
\left[ E, F^{(n+1)}  \right] = F^{(n)} \left( q^{-n} K - q^n K^{-1}  \right) \text{ and }
F^{(n)} F^{(m)} = \qbin{n+m}{n}_q F^{(n+m)} .
\]
Together with relations from Definition \ref{Uqnaif}, they complete a presentation of $\UqhL$. 

$\UqhL$ inherits a Hopf algebra structure, making its category of modules monoidal. The coproduct is given by:
\[
\Delta(K) = K \otimes K \text{ , } \Delta(E) = E \otimes K + 1 \otimes E , \text{ and } \Delta(F^{(n)}) = \sum_{j=0}^n q^{-j(n-j)}K^{j-n} F^{(j)} \otimes F^{(n-j)}. 
\]
\end{rmk}
%
%\begin{prop}
%The algebra $\UqhL$ admits the following set as an $\Laurent_0$-basis:
%\[
%\left\lbrace K^l E^m F^{(n)} , l \in \BZ, m,n \in \BN \right\rbrace .
%\]
%\end{prop}

\subsubsection{Verma modules and braiding}\label{VermaBraiding}

Now we define a special family of universal objects in the category of $\Uq$-modules, we express their presentation in the special case of $\UqhL$ and we give a braiding for this family of modules. Namely, the {\em Verma modules} are infinite dimensional modules depending on a parameter. Again, we work with this parameter as a variable included in an integral ring, letting $\Laurent := \BZ \left[ q^{\pm 1} , s^{\pm 1} \right]$. %In \cite{JK}, they give an explicit presentation for the integral Verma-module of $\UqhL$, that we recall here.

\begin{defn}[Verma modules for $\UqhL$]\label{GoodVerma}
Let $V^{s}$ be the Verma module of $\UqhL$. It is the infinite $\Laurent$-module, generated by vectors $\lbrace v_0, v_1 \ldots \rbrace$, and endowed with an action of $\UqhL$, generators acting as follows:
\[
K \cdot v_j = s q^{-2j} v_{j} \text{ , } E \cdot v_j = v_{j-1} \text{ and } F^{(n)} v_j = \left( \qbin{n+j}{j}_q \prod_{k=0}^{n-1} \left( sq^{-k-j} - s^{-1}q^{j+k} \right) \right) v_{j+n} .
\]
\end{defn}

\begin{rmk}[Weight vectors]\label{weightdenomination}
We will often make implicitly the change of variable $s := q^{\alpha}$ and denote $V^s$ by $V^{\alpha}$. This choice made to use a practical and usual denomination for eigenvalues of the $K$ action (which is diagonal in the given basis). Namely we say that vector $v_j$ is {\em of weight $\alpha - 2j$}, as $K \cdot v_j = q^{\alpha - 2j} v_j$. The notation with $s$ shows a  strictly speaking integral structure of Laurent polynomials. 
\end{rmk}

\begin{defn}[$R$-matrix, {\cite[(21)]{JK}}]\label{goodRmatrix}
Let $s=q^{\alpha}$ , $t=q^{\alpha'}$. The operator $q^{H \otimes H /2}$ is the following:
\[
q^{H \otimes H /2}:
\bfct
V^{s} \otimes V^{t} & \to & V^{s} \otimes V^{t}  \\
v_i \otimes v_j & \mapsto & q^{(\alpha - 2i)(\alpha'-2j)} v_i \otimes v_j 
\efct .
\]
We define the following R-matrix:
\[
R : q^{H \otimes H/2} \sum_{n=0}^\infty q^{\frac{n(n-1)}{2}} E^n \otimes F^{(n)} .
\]
It is not a well defined object as the sum is infinite, but whenever it is applied to Verma modules vectors, the sum becomes finite (thanks to the {\em locally nilpotence} of $E$). It justifies the fact that it is a well defined operator on tensor product of Verma modules. .%which will be well defined as an operator on Verma modules in what follows. 
\end{defn}

\begin{prop}[{\cite[Theorem~7]{JK}}]\label{GoodBnrep}
Let $V^s$ and $V^t$ be Verma modules of $\UqhL$ (with $s=q^{\alpha}$ and $t=q^{\alpha'}$). Let $\RR$ be the following operator:
\[
\RR: q^{-\alpha \alpha' /2} T \circ R 
\]
where $T$ is the twist defined by $T(v\otimes w ) = w \otimes v$. Then $\RR$ provides a braiding for $\UqhL$ integral Verma modules. 
Namely, the morphism:
\[
Q: \bfct
\Laurent_1\left[ \Bn \right] & \to & \End_{\Laurent, \UqhL} \left({V^s}^{\otimes n}\right)  \\
\sigma_i & \mapsto & 1^{\otimes i-1} \otimes \RR \otimes 1^{\otimes n-i-2}
\efct
\]
is an $\Laurent$-algebra morphism. It provides a representation of $\Bn$ such that its action commutes with that of $\UqhL$. 
\end{prop}

\begin{rmk}\label{coloredquantum}
One can consider a braid action over $V^{s_1} \otimes \cdots \otimes V^{s_n}$ so that the morphism $Q$ is well defined but becomes multiplicative (i.e. algebra morphism) only when restricted to the pure braid group $\PBn$, so to be an endomorphism. See \cite[Appendix]{Jules2} for a detailed explanation. The latter is encouraging for generalizing the present work for colored Jones polynomials of knots to links.  %Then one can consider the induced representation of $\Bn$, see Definition \ref{inducedpure}, or restrict to a representation of the colored braid groupoid. 
\end{rmk}

\begin{defn}[Weight spaces, {\cite[Remark~5.13]{Jules1}}]
The representation of $\Bn$ on ${V^s}^{\otimes n}$ splits as follows:
\[
{V^s}^{\otimes n} \simeq \bigoplus_{r \in \BN} V_{n,r},
\]
where $V_{n,r}:=\ker\left(K-s^n q^{-2r}\Id \right)$.
\end{defn}

In Section \ref{cJonesSetup} we will define the colored Jones invariant for knots out of these representations for braid groups. 

\subsection{Homological representations}\label{HXrmodel}

\subsubsection{Configuration spaces}

\begin{defn}\label{configspaceofpoints}
Let $r\in \BN$, $n \in \BN$, $D\in \BC$ be the unit disk, and $\left\lbrace w_1 , \ldots , w_n \rbrace\right. \in D^n$ points lying on the real line in the interior of $D$. Let $D_n = D \setminus \left\lbrace w_1 , \ldots , w_n \rbrace\right.$ be the unit disk with $n$ punctures. Let:
\[
\Conf_r(D_n):= \left\lbrace (z_1 , \ldots , z_r ) \in (D_n)^r \text{ s.t. } \begin{array}{c} z_i \neq z_j \forall i,j  \end{array} \right\rbrace
\]
be the configuration space of points in the punctured disk $D_n$. 
We define the following space:
\begin{eqnarray}\label{NotConfig}
X_r(w_1 , \ldots , w_n) & = &  \Conf_r(D_n) \Big/ \Sk_r. 
\end{eqnarray}
to be the space of {\em unordered} configurations of $r$ points inside $D_n$, where the permutation group $\Sk_r$ acts by permutation on coordinates.
%The points $w_1 , \ldots , w_n$ will always be chosen so that they lie in the interior of $D$ and over the real line.
\end{defn} 

When no confusion arises in what follows, we omit the dependence in $w_1 , \ldots , w_n$ to simplify notations. All the following computations rely on a choice of base point that we fix from now on.

\begin{defn}[Base point]\label{basepoint}
Let ${\pmb \xi^r}= \lbrace  \xi_1 , \ldots , \xi_r \rbrace$ be the base point of $X_r$
%\[
%{\pmb \xi^r} 
%\]
chosen so that $ \xi_i \in \partial D_n$ $\forall i$ with negative imaginary parts, and so that:
\[
\Re(w_0)< \Re(\xi_r)< \Re(\xi_{r-1})< \ldots <\Re(\xi_1) <\Re(w_1).  
\]
\end{defn}
We illustrate the disk with chosen points in the following figure.

\begin{equation*}
\begin{tikzpicture}[scale=0.7]
\node (w0) at (-3,0) {};
\node (w1) at (1,0) {};
\node (w2) at (2,0) {};
\node[gray] at (2.8,0) {\ldots};
\node (wn) at (3.4,0) {};
%\node (wn1) at (3,1) {};
%\node (wn) at (5,1) {};

%\node (x0) at (-4.8,-3) {};
%%\node at (-4.7,-3) {$\ldots$};
%\node  (x1) at (-4.3,-3) {};
%\node  (x2) at (-4.25,-3) {};
%\node  (xn1) at (,-3) {};
%\node  (xn) at (4,-3) {};

%\draw[dashed,gray] (-5,0) -- (5,0) -- (5,-3);
\draw[thick,gray] (4,2) -- (-3,2) -- (-3,-2) -- (4,-2) -- (4,2);% node[right] {$\partial D_n$};

%\node[above,red] at (-3.5,0) {$k_0$};
%\node[above,red] at (0,0) {$k_1$};
%\node[above,red] at (4,0) {$k_{n-1}$};

\node[below,red] at (-1,-2) {$\xi_r$};
\node[below,red] at (-0.3,-2) {$\xi_{r-1}$};
%\node[below,red] at (0.15,-2) {$\xi_{r''}$};
\node[below=3pt,red] at (0.3,-2) {\small $\ldots$};
\node[below,red] at (0.8,-2) {$\xi_{1}$};

\node at (-1,-2)[red,circle,fill,inner sep=1pt]{};
\node at (-0.3,-2) [red,circle,fill,inner sep=1pt]{};
%\node at (0.15,-2) [red,circle,fill,inner sep=1pt]{};
\node at (0.8,-2) [red,circle,fill,inner sep=1pt]{};

%
%\node[red] at (-3.8,-1) {$\ldots$};
%\node[red] at (-1.2,-0.9) {$\ldots$};
%\node[red] at (0.25,-0.9) {$\ldots$};
%\node[red] at (2.9,-0.5) {$\ldots$};

\node[gray] at (w0)[left=5pt] {$w_0$};
\node[gray] at (w1)[above=5pt] {$w_1$};
\node[gray] at (w2)[above=5pt] {$w_2$};
%\node[gray] at (wn1)[above=5pt] {$w_{n-1}$};
\node[gray] at (wn)[above=5pt] {$w_n$};
\foreach \n in {w1,w2,wn}
  \node at (\n)[gray,circle,fill,inner sep=3pt]{};
\node at (w0)[gray,circle,fill,inner sep=3pt]{};
\end{tikzpicture} .
\end{equation*}

%In what follows, distances between the $\xi_i$'s may be deformed in drawings but the order on real parts remains the important fact. 
We draw a square boundary for the disk, in order for the reader not to confuse it with arcs we will be drawing inside.

%\end{defn}

We give a presentation of $\pi_1(X_r, {\pmb \xi^r})$ as a braid subgroup ({\em the mixed braid group}), which can be deduced from the one given in the introduction of \cite{Z1}. See \cite[Remark~2.2~and~Example~2.3]{Jules1} for the correspondence between braids and elements of $\pi_1(X_r, {\pmb \xi^r})$. 

\begin{rmk}[{\cite[Remark~2.2]{Jules1}}]\label{pi_1X_r}
The group $\pi_1(X_r, {\pmb \xi^r})$ is isomorphic to the subgroup of $\CB_{r+n}$ generated by:
\[
\langle \sigma_1 , \ldots , \sigma_{r-1}, B_{r,1} , \ldots , B_{r,n} \rangle 
\]
where the $\sigma_i$ ($i=1,\ldots ,r-1$) are standard generators of $\CB_{r+n}$, and $B_{r,k}$ (for $k=1,\ldots ,n$) is the following pure braid:
\[
B_{r,k} = \sigma_{r} \cdots \sigma_{r+k-2} \sigma_{r+k-1}^2 \sigma_{r+k-2}^{-1} \cdots \sigma_{r}^{-1} .
\]
\end{rmk}

%To see the correspondence between loops in $X_r$ and generators of the above braid subgroup we draw two examples. 

%\begin{ex}
%Two types of braid generators for $\pi_1(X_r, {\pmb \xi^r})$ are given in Remark \ref{pi_1X_r}, which correspond to two types of loops generating $\pi_1(X_r, {\pmb \xi^r})$. We give examples for both kinds.
%
%\begin{itemize}
%\item The braid $\sigma_1$ corresponds to a loop swapping $\xi_r$ and $\xi_{r-1}$ letting other base point coordinates fixed. This can be seen by drawing the movie of the loop in Figure \ref{pi1sigma}.
%\begin{figure}[h!]
%\begin{center}
%\def\svgwidth{0.5\columnwidth}
%%\def\svgscale{0.3}
%\input{pi1sigma.pdf_tex}
%\caption{Generator $\sigma_1$. \label{pi1sigma}}
%\end{center}
%\end{figure}
%
%\item The braid $B_{r,k}$ for $k \in \lbrace 1 , \ldots , n \rbrace$ corresponds to $\xi_1$ running once around $w_k$ before going back keeping other base point coordinates fixed. The correspondence in terms of standard braid generators can be seen by drawing the movie of this loop in Figure \ref{pi1B}.
%
%\begin{figure}[h!]
%\begin{center}
%\def\svgwidth{0.5\columnwidth}
%%\def\svgscale{0.3}
%\input{pi1B.pdf_tex}
%\caption{Generator $B_{r,k}$ \label{pi1B}}
%\end{center}
%\end{figure}
%\end{itemize}
%\end{ex}

Using this set up, we define the local system of interest. 

\begin{defn}[Local system $L_r$.]\label{localsystXr}
Let $L_r(w_1,\ldots ,w_n)$ be the local system defined by the following algebra morphism:
\[
\rho_r : \bfct
\BZ\left[ \pi_1(X_r, {\pmb \xi^r}) \right] & \to & \Laurent := \BZ \left[ s^{\pm1} , t^{\pm 1}  \right]\\
\sigma_i & \mapsto & t \\
B_{r,k} & \mapsto & s^2 . \\
\efct
\]
In what follows we will use the notation $q^{\alpha}:=s$. Using this notation, the morphism becomes:
\[
\rho_r : \bfct
\BZ\left[ \pi_1(X_r, {\pmb \xi^r}) \right] & \to & \Laurent := \BZ \left[ q^{\pm \alpha} , t^{\pm 1}  \right]\\
\sigma_i & \mapsto & t \\
B_{r,k} & \mapsto & q^{2\alpha} . \\
\efct
\]
(We may sometimes omit the dependence in $(w_1,\ldots,w_n)$ in the notations to simplify them.)
\end{defn}

\subsubsection{Homology with local coefficients}

\begin{defn}[{\cite[Definition~2.6]{Jules1}}]\label{defofH}
Let $r \in \BN$, and let $w_0 = -1$ be the leftmost point on the boundary of the disc $D_n$, we define the following set:
\[
X_r^-(w_1 , \ldots , w_n) = \left\lbrace \left\lbrace z_1 , \ldots , z_r \right\rbrace \in X_r(w_1 , \ldots , w_n) \text{ s.t. } \exists i, z_i=w_0 \right\rbrace .
\]
We let {\em $\Hlf$} denote the homology of locally finite chains, and we use the following notations for relative homology modules with local coefficients in the ring $\Laurent$:
\[
\Hrelm_r := \Hlf _r \left( X_r, X^{-}_r ; \Laurent \right).
\] 
See next remark for precision of such construction. 
\end{defn}

\begin{rmk}
We recall how the above homology modules are constructed, namely we work with the following homology theories:
\begin{itemize}
%We work with the following homology theories:
\item the {\em locally finite} version of the singular homology, see \cite[Appendix]{Jules1}.
\item the homology of the pair $(X_r,X_r^{-})$, see \cite[Appendix]{Jules1}. 
\item the local ring of coefficients $\Laurent$. Let $\rho_r$ be the morphism from Definition \ref{localsystXr}, the local ring action can be seen as the homology associated with the following chain complex (two equivalent definitions):
\[
C_{\bullet}(\widetilde{X_r} , \BZ ) \otimes_{\pi_1(X_r)} \Laurent \simeq C_{\bullet} (\widehat{X_r}) ,
\]
where $\widetilde{X_r}$ is the universal cover of $X_r$, $\widehat{X_r}$ the one associated with the kernel of the morphism $\rho_r$. Then in the above chain complexes, on the left the action of $\pi_1(X_r)$ on the right tensor is given by $\widehat{\rho_r}$; while on the right term, the chain complex of $\widehat{X_r}$ is endowed with the action of $\Laurent$ by deck transformations. 
\item The homology of locally finite chains is, in our case, isomorphic to the {\em Borel-Moore homology} that can be defined as follows:
\[
H_{\bullet}^{\mathrm{BM}} (X) = \varprojlim H_{\bullet} \left( X , X \setminus A \right) 
\]
where the inverse limit is taken over all compact subsets $A$ of $X$. For the relative version, see \cite[Appendix]{Jules1}. 
\end{itemize}
Following computations depend on a choice of lift for the base point ${\pmb \xi^r}$ that we make now and we denote it $\widehat{{\pmb \xi^r}}$. 
\end{rmk}

We define classes in $\Hrelm_r$.

\begin{defn}[Code sequences diagrams]\label{codesequences}
Let $(k_0 , \ldots,  k_{n-1})$ such that $\sum k_i = r$. we define $U(k_0 , \ldots , k_{n-1})$ to be the following diagram:
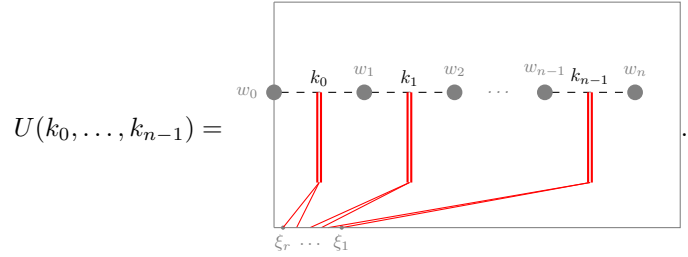
\begin{equation*}
U(k_0 , \ldots , k_{n-1}) = \vcenter{\hbox{\begin{tikzpicture}[scale=0.6, every node/.style={scale=0.7},decoration={
    markings,
    mark=at position 0.5 with {\arrow{>}}}
    ]
\node (w0) at (-5,0) {};
\node (w1) at (-3,0) {};
\node (w2) at (-1,0) {};
\node[gray] at (0.0,0.0) {\ldots};
\node (wn1) at (1,0) {};
\node (wn) at (3,0) {};

\draw[dashed] (w0) -- (w1) node[midway,above] (k0) {$k_0$};
\draw[dashed] (w1) to node[midway,above]  (k1) {$k_1$} (w2);
%\draw[dashed] (w0) to node[pos=0.85] (k2) {$k_{n-2}$} (wn1);
\draw[dashed] (wn1) to node[midway,above]  (k3) {$k_{n-1}$} (wn);

\node[gray] at (w0)[left=5pt] {$w_0$};
\node[gray] at (w1)[above=5pt] {$w_1$};
\node[gray] at (w2)[above=5pt] {$w_2$};
\node[gray] at (wn1)[above=5pt] {$w_{n-1}$};
\node[gray] at (wn)[above=5pt] {$w_n$};
\foreach \n in {w1,w2,wn1,wn}
  \node at (\n)[gray,circle,fill,inner sep=3pt]{};
\node at (w0)[gray,circle,fill,inner sep=3pt]{};

\coordinate (c) at (-2,-2);

\draw[double,thick,red] (k0) -- (k0|-c);
\draw[double,thick,red] (k1) -- (k1|-c);
%\draw[double,thick,red] (k2) -- (0.05,-3);
\draw[double,thick,red] (k3) -- (k3|-c);
%
%
%\draw[dashed,gray] (-5,-3) -- (3.5,-3);
%\draw[dashed,gray] (3.5,-3) -- (3.5,-4);

\node[gray,circle,fill,inner sep=0.8pt] at (-4.8,-3) {};
\node[below,gray] at (-4.8,-3) {$\xi_r$};
\node[below=5pt,gray] at (-4.2,-3) {$\ldots$};
\node[gray,circle,fill,inner sep=0.8pt] at (-3.5,-3) {};
\node[below,gray] at (-3.5,-3) {$\xi_1$};

%\draw[red] (-4.8,-3) -- (-4.5,-3.8);
%\draw[red] (-3.5,-3) -- (-3,-3.9);

\draw[red] (-4.8,-3) -- (k0|-c);
\draw[red] (-4.5,-3) -- (k0|-c);
\draw[red] (-3.5,-3) -- (k3|-c);
\draw[red] (-3.8,-3) -- (k3|-c);
%\draw[red] (-3.9,-4) -- (0.05,-3);
\draw[red] (-4.2,-3) -- (k1|-c);
\draw[red] (-3.95,-3) -- (k1|-c);
%\node[red] at (-2,-3.4) {$\ldots$};

\draw[gray] (-5,0) -- (-5,2);
\draw[gray] (-5,0) -- (-5,-3);
\draw[gray] (-5,-3) -- (4,-3) -- (4,2) -- (-5,2);% node[right] {$\partial D_n$};

\end{tikzpicture}}} .%\text{ ~``multi - arcs'' }
\end{equation*}

\end{defn}

For $U(k_0 , \ldots , k_{n-1})$ defined above, let:
\[
\phi_i : I_i \to D_n
\]
be the embedding of the dashed black arc number $i$ indexed by $k_{i-1}$, where $I_i$ is a unit interval.
Let $\Delta^k$ be the standard (open) $k$ simplex:
\[
\Delta^k = \lbrace 0 < t_1 < \cdots < t_k < 1 \rbrace 
\]
for $k \in \BN$.
For all $i$, we consider the map $\phi^{k_{i-1}}$:
\[
\phi^{k_{i-1}}: \bfct
\Delta^{k_{i-1}} & \to & X_{k_{i-1}} \\
(t_1, \ldots , t_{k_{i-1}} ) & \mapsto & \lbrace \phi_i(t_1) , \ldots, \phi_i(t_{k_{i-1}}) \rbrace
\efct
\]
which is a singular locally finite $(k_{i-1})$-chain and moreover a cycle in $X_{k_{i-1}}$ (see \cite[Section~3.1]{Jules1}). %One can think of the image of the simplex $\Delta^{k_{i-1}}$ to be the space of configurations of $k_{i-1}$ points inside the dashed arc. It provides a locally finite cycle as going to a face of the simplex corresponds to going to a collision between either two configuration points, either a configuration point with a puncture. Namely, points in the boundary of the simplex are removed points of the configuration space $X_r$, these simplices are closed submanifold going to infinity, so that they are locally finite cycles, see the Appendix. There is a cycle associated with each dashed arc, so that by considering the product of maps $(\phi^{k_{0}},\ldots,\phi^{k_{n-1}})\in \Conf_r(D_n)$ which is naturally sent to $X_r$, one  generalizes this fact by associating an $r$-cycle of $X_r$ with each object $X(k_0 , \ldots , k_{n-1})$, see following Remark \ref{chainwithdisjointsupport}. This shows how the union of dashed arcs defines a class in the homology with coefficient in $\BZ$.

To get a class in the homology with $\Laurent$ coefficients, one may choose a lift of the chain to the cover $\widehat{X_r}$ associated with the morphism $\rho_r$. We do so using the red handles of $U(k_0 , \ldots , k_{n-1})$ (the union of red paths) with which is naturally associated a path:
\[
{\bf h}=\lbrace h_1,\ldots,h_r \rbrace: I \to X_r
\]
joining the base point $\pmb{\xi}$ and the $r$-chain assigned to the union of dashed arcs. At the cover level ($\widehat{X_r}$) there is a unique lift $\widehat{{\bf h}}$ of ${\bf h}$ that starts at $\widehat{{\pmb \xi}}$. The lift of $U(k_0, \ldots , k_{n-1})$ passing by $\widehat{\pmb h} (1)$ defines a cycle in $\Crelm_r$, and we still call $U(k_0 , \ldots , k_{n-1})$ the associated class in $\Hrelm_r$ as we will only use this class out of the original object. 

\begin{defn}[Code sequences]\label{multiarcsclass}
Following the above construction, we naturally assign a class $U(k_0 , \ldots , k_{n-1}) \in \Hrelm_r$ for any $n$-tuple such that $\sum k_i = r$. This class is called a {\em code sequence}. 
\end{defn}

\begin{prop}[Code sequences generate the homology, {\cite[Proposition~3.6]{Jules1}}]\label{homologystructure}
Let $r \in \BN$, the homology of the pair $(X_r, X_r^{-})$ has the following structure:
\begin{itemize}
\item The module $\Hrelm_r$ is free over $\Laurent$. 
\item The set of code sequences:
\[
\CU := \lbrace U(k_0, \ldots, k_{n-1}) \text{ s.t. } \sum k_i = r \rbrace
\]
yields a basis of $\Hrelm_r$. 
\item The module $\Hrelm_r$ is the only non vanishing module of the complex $\Hlf_{\bullet}\left( X_r , X_r^{-}; \Laurent \right) $.
\end{itemize}
\end{prop}

\subsubsection{Action of the braid groups}

\begin{defn}
The braid group on $n$ strands is the mapping class group of $D_n$.
\[
\Bn = \Mod(D_n) = \Homeo(D_n, \partial D) \big/ {\Homeo}_0(D_n, \partial D).
\]
\end{defn}

\begin{rmk}\label{halfDehntwist}
This definition is isomorphic to the Artin presentation of the braid group (Definition \ref{Artinpres}) by sending generator $\sigma_i$ to the mapping class of the half Dehn twist swapping punctures $w_i$ and $w_{i+1}$. %The {\em pure braid group} $\PBn$ is defined to be braids fixing the punctures pointwise. 
\end{rmk}

\begin{lemma}[Lawrence representations, {\cite[Lemma~6.33]{Jules1}}]\label{Lawrencerep}
For all $r,n \in \BN$, the modules $\Hrelm_r$ are endowed with an action of the braid group $\Bn$. 
\end{lemma}
\begin{proof}
It is Lawrence construction of braid groups representations \cite{Law}. See \cite[Lemma~6.33]{Jules1} for this precise Lemma. The representations are constructed as follows (sketch of proof).
\begin{itemize}
\item Let $S_i$ be the Dehn twist associated with the standard Artin generator $\sigma_i$ of $\Bn$, for $i \in \lbrace 1, \ldots, n-1 \rbrace$ (see Remark \ref{halfDehntwist}).
\item The homeomorphism $S_i$ extends to $X_r$ coordinate by coordinate.
\item The action of $S_i$ on $X_r$ lifts to $\Hrelm_r$. 
\item Define the action of $\sigma_i$ on $\Hrelm_r$ to be that of $S_i$. It defines a representation of $\Bn$ on $\Hrelm_r$ by isotopy invariance of the homological action. 
\end{itemize}
\end{proof}

The above representations are often called Lawrence(-like) representations. 

\begin{defn}\label{RepBn}
We denote $\Rhom$ the Lawrence representation of the braid group $\Bn$ on $\Hrelm_r$.
\[
\Rhom : \Laurent \left[ \Bn \right] \to \End_{\Laurent} \left( \Hrelm_r \right)
\]
\end{defn}

\subsection{Quantum representations from homology}\label{SummaryJules1}

In the two previous subsections we have introduced two families of representations for braid groups. In Section \ref{qVermaRep} the ones over tensor products of Verma modules, in Section \ref{HXrmodel} Lawrence ones upon homologies of configuration spaces. It turns out that they are related by the following theorem.

\begin{theorem}[{\cite[Theorem~2,~3]{Jules1}}]\label{Jules1Theorems}
There exists an isomorphism of $\Laurent$-modules:
\[
\CH:= \bigoplus_{r\in \BN} \Hrelm_r \to \left({V^{\alpha}}\right)^{\otimes n} = \bigoplus_{r \in \BN} V_{n,r}
%A(k_0, \ldots ,k_{n-1}) & \mapsto & v_{k_0} \otimes \cdots \otimes v_{k_{n-1}} \text{ ( } \forall (k_0, \ldots ,k_{n-1}) \in \BN^n \text{ ) }
\]
respecting the grading and which is also an isomorphism of $\Bn$ representations. 
\end{theorem}

Theorem \cite[Theorem~2]{Jules1} also shows that the above isomorphism preserves the action of $\UqhL$ defines on tensor products of Verma modules by the coproduct and on the homological side in \cite[Section~6.1]{Jules1}. This isomorphism is used to compute Jones polynomial from homologies of configuration spaces in next section.

\section{Colored Jones polynomials and Lefschetz numbers}\label{Sectionofmaintheorem}

The purpose of this section is to state and prove Theorem \ref{JonesisLefschetz}. It is Theorem \ref{maintheoremIntro} from the intro, it gives a formula relating abelianized Lefschetz numbers and colored Jones polynomials. In Section \ref{surveyLefschetz} we give a brief survey on the generalized Lefschetz number theory, involving its relations with Nielsen fixed point theory. In Section \ref{cJonesSetup} we define the family of colored Jones polynomials for knots from finite dimensional and simple representation of the algebra $\Uq$. We then show how to compute these knot invariants out of Verma modules instead. In Section \ref{subsectionofmaintheorem} we state and prove Theorem \ref{JonesisLefschetz} relating the two notions introduced: abelianized Lefschetz numbers and colored Jones polynomials. 

\subsection{Generalized Lefschetz numbers}\label{surveyLefschetz}

We make recalls about the {\em generalized Lefschetz number} (first introduced in \cite{Hu}), the {\em abelianized Nielsen number} (\cite[Section~1.(B)]{HJ}) and conclude with the {\em mod $K$ Nielsen numbers} (for $K$ a normal subgroup of the fundamental group) (\cite[Chapter~III.2]{J} or \cite[Part~1]{Nielsentheory}).

We follow \cite[Section~1]{HJ} for the definitions concerning generalized and abelianized Nielsen theories. Let $X$ be a compact connected topological space admitting a universal covering space, and $f$ be a homeomorphism of $X$. The fixed point set $Fix(f) = \lbrace x\in X \text{ s.t. } f(x) = x \rbrace$ splits into the union of {\em Nielsen fixed point classes}.

\begin{defn}
Two fixed points $x$  and $y$ of $f$ are in the same Nielsen class if there exists a path $\lambda$ joining them and such that $\lambda$ is homotopic (rel. endpoints) to its image under $f$. Nielsen classes are isolated so that their fixed point indexes are well defined. 
\end{defn}

Suppose $f$ fixes the base point $x_0$ of $X$ so that we can define $f_{\pi}$ to be the action of $f$ on $\pi_1(X)$ (otherwise the choice of a path $w$ from $x_0$ to its $f$-image is necessary to define $f_{\pi}$). 

\begin{defn}
We say that $\alpha, \beta \in \pi_1(X)$ are in the same $f$-Reidemeister class if there exists $\gamma \in \pi_1(X)$ such that $\alpha = f_{\pi}(\gamma) \beta \gamma^{-1}$. We call $\pi_{R}$ the group of such {\em $f$-conjugacy classes} and $\BZ \pi_{R}$ its group ring. 

Let $x$ be a fixed point and $c$ be a path from $x_0$ to $x$. The Reidemeister class of the loop $(f \circ c ) c^{-1}$ is independent of the choice of $c$ and called the coordinate of $x$. 
\end{defn}

\begin{prop}[{\cite[Section~1.(A)]{HJ}}]
Two fixed points are in the same Nielsen class iff they have the same coordinates.  The involved Reidemeister class is thus the {\em Nielsen class coordinate}. 
\end{prop}

\begin{defn}
The {\em generalized Lefschetz number} of $f$ is defined to be:
\[
\CL_{R} (f) = \sum_{\left[ \alpha \right] \in \pi_R} i_{\left[ \alpha \right]  }\left[ \alpha \right] \in \BZ \pi_{R}
\]
where $\left[ \alpha \right]$ is a class in $\pi_R$ and $i_{\left[ \alpha \right]  }$ the index of fixed points in the corresponding Nielsen class.
\end{defn}

The following fact makes the link with the traditional notions of Lefschetz numbers and Nielsen numbers and can be taken as definition for these objects in the present paper. 

\begin{defn}
The usual Nielsen number and Lefschetz number are the following numbers:
\begin{align*}
N(f) & = \sharp \lbrace \left[ \alpha \right] \in \pi_R \text{ s.t. } i_{\left[ \alpha \right]} \neq 0 \rbrace \\
L(f) & = \sum_{\left[ \alpha \right] \in \pi_R} i_{\left[ \alpha \right]} .
\end{align*}
\end{defn}

\begin{theorem}[{\cite[Theorem~1.13]{Hu}},{\cite[Section~1.(A)]{HJ}}]
Suppose $X$ is a CW complex. A cellular decomposition of $X$ lifts to cells of $\widetilde{X}$, its universal covering space. These cells constitute a $\BZ \pi_1(X)$ basis of the cellular chain complex of $\widetilde{X}$. Let $f$ be a chain map and $\widetilde{f}$ be a lift of $f$ to $\widetilde{X}$ (considering a lift of the base point). Then:
\[
\CL_R(f) = \sum_{q} (-1)^q \left[ \Tr(\widetilde{f}, C_q(\widetilde{X}, \BZ \pi_{R}) ) \right]_R \in \BZ \pi_{R}.
\]
where $\left[ \Tr(\widetilde{f}, C_q(\widetilde{X}, \BZ \pi_{R}) ) \right]_R$ is the $\pi_R$-class of the trace of the action of $\widetilde{f}$ over the $\BZ \pi_1(X)$-module $C_q(\widetilde{X}, \BZ \pi_{R})$.
\end{theorem}

The proof of the above theorem relies on the lifting property of maps to the universal cover. See a sketch of proof in \cite[Theorem~2.2]{Ha}. This proof can be adapted to Borel-Moore chain complexes over punctured manifold if $f$ is a proper map, as Borel-Moore corresponds to an inductive limit of chain complexes over compact topological spaces. 

One can adapt the theory of generalized Lefschetz numbers to other covering spaces having the good lifting properties. Indeed, the theory of generalized Lefschetz number corresponds to working with the universal covering space but has analogs working with other ones. For instance the maximal abelian cover cover gives rise to the notion of abelianized Nielsen number. 

\begin{defn}[{\cite[Section~1.(A)]{HJ}}]
Two fixed points are in the same {\em homological Nielsen class} if there exists a path $\lambda$ joining them and such that $\lambda$ is homologous (rel. endpoints) to its image under $f$. The {\em abelianized Nielsen number} $\CN^{ab}(f)$ is the number of homological Nielsen classes with non-zero index. 
\end{defn}

We consider the following morphism:
\[
\mu : H_1(X) \to H = \Coker \left(1-f_* \in \End(H_1(X)) \right) .
\]
The coordinate of a homological Nielsen class is the $H$-class of the loop $(f\circ c)c^{-1}$ where $c$ is a path from $x_0$ to a fixed point of the class. We recall that if $\widehat{X}$ is the maximal abelian cover of $X$ (associated with the Hurewicz representation: $\pi_1(X) \to H_1(X)$), the chain complex $C_{\bullet}(\widehat{X})$ is endowed with a $\BZ H_1(X)$-action. As for the Nielsen number, there is a trace formula computing the abelianized Nielsen number.

%%The coordinate of a homological Nielsen class is the element of H represented by
%%the loop w(f 0 c)c- 1 , where e is a path from Xo to a fixed point in the class. The
%%abelianized Reidemeister trace LH(f) is an element ofthe group-ring 7I..H, the coefficient
%%of h E H being the index of the homological Nielsen class with coordinate h. Thus
%%LH(I) = LR(f)Ab and, in the setting of cellular computation in (A),
%%r.
%%LH(f) = L(-l)qtrFH,q E 7I..H
%%q
%%where ·PH,q := p,Ptb are 71..H-matrices. 

\begin{prop}[{\cite[Section~1.(B)]{HJ}}]\label{abelianizedReidemeisterTrace}
The abelianized Reidemeister trace $\CL_H(f)$ is an element of $\BZ H$ such that the coefficient of $h \in H$ is the index of the homological class with coordinate $h$. Thus $\CL_H(f)=\CL_R(f)^{ab}$ (the generalized Lefschetz number abelianized), and it satisfies the following trace formula:
\[
\CL_H(f) = \sum_{q} (-1)^q \mu \left[ \Tr(\widetilde{f}, C_q(\widehat{X}, \BZ H_1(X)) ) \right] \in \BZ H.
\]
The number of non-zero terms is $\CN^{ab}(f)$.
\end{prop}

\begin{rmk}\label{lowerboundandlefschetz}
Although the abelianized Nielsen number is less refined than the Nielsen number, it enjoys the following properties:
\begin{itemize}
\item $\CN^{ab}(f) \leq N(f)$ .
\item $L(f) = \sum i_{\left[ \alpha \right]_H}$ where the terms in the sum are the indexes of homological classes associated with $\alpha$. 
\end{itemize}
\end{rmk}

One could work with a covering space corresponding to a normal subgroup of the $\pi_1$ and obtain the same kind of result. Namely, the abelianized Nielsen theory corresponds to the the $mod$ $K$ Nielsen theory (\cite[Chapter~III.2]{J}) for $K = \left[ \pi_1(X) , \pi_1(X) \right]$. We will precise the group $K$ of interest to us in the context of the following subsection.

\begin{theorem}[{\cite[Theorem~2.5]{J}}]
The Nielsen number, abelianized Nielsen number, and the $mod$ $K$ Nielsen numbers are homotopy invariants. 
\end{theorem} 

\subsection{Colored Jones polynomials}\label{cJonesSetup}

For defining colored Jones polynomials of knots, we need standard finite dimensional simple modules of $\Uq$. In this section, we recover the trace of the action of a braid on these finite dimensional modules inside that on Verma modules. We follow \cite[Section~5]{JK} to do so. This will allow us to compute the colored Jones polynomial from the action of braids on specific tensor products of Verma modules. Let $l\in \BN^*$ be an integer, and consider the specialization ring morphism $\aug^l : \BZ \left[ s^{\pm 1} , q^{\pm 1} \right] \to \BZ \left[ q^{\pm 1} \right]$ that sends $s \mapsto q^l$. Let $V^s$ be the $\UqhL$ Verma module defined in Definition \ref{GoodVerma}, spanned by vectors $\lbrace v_0 , v_1 , \ldots \rbrace$, and:
\[
\bfVl = V^s \otimes_{s=q^l} \BZ \left[ q^{\pm 1} \right] ,
\]
be the tensor product provided by morphism $\aug^l$.

\begin{rmk}[standard finite simple modules out of Verma modules, {\cite[Section~5]{JK}}]\label{truncatedVerma}
Let $S^l \in \bfVl$ be the submodule spanned by vectors $\lbrace v_0 , \ldots , v_l \rbrace$. This submodule is equipped with an action of generators $F^{(k)}$ that truncates as \cite[Relation~(47)]{JK}, which is straightforward from the formula in Definition \ref{GoodVerma}. Namely:
\begin{equation}\label{FtruncatesApplicationtoknots}
F^{(k)} v_j = 0 \text{ if } k+j > l .
\end{equation}
The latter ensures that $S^l \in \bfVl$ is also a submodule regarding the $\Uq$-structure. It is equivalent to the $(l+1)$ dimensional standard simple module (precisely defined in \cite[Section~VIII.3]{Kas}).

Let $n \in \BN^*$, from the latter together with Relations (48) and (49) of \cite{JK}, $({S^l})^{\otimes n} \subset {\bfVl}^{\otimes n}$ is a sub-representation of the braid group $\Bn$, see Definition \ref{GoodBnrep} for the definition of the representation. Namely, the braid group representations can be specialized and restricted to finite dimensional modules. Moreover the obtained representations of the braid group $\Bn$ over $({S^l})^{\otimes n}$ are equivalent to those obtained from the standard construction of the $(l+1)$-dimensional simple representations of $\Uq$. 
\end{rmk}

\begin{Not}
Let $l,n \in \BN^*$, We denote:
\[
V_{n,r} := \ker \left( K - q^{nl-2i} 1 \right) \in ({\bfVl})^{\otimes n},
\]
and $S_{n,l} := \bigoplus_{r=0}^{nl} V_{n,r}$ be the direct sum of the first $nl$ weight spaces. The latter is an $\Laurent$-module spanned by the following set:
\[
\CB_{S_{nl}} = \left\lbrace v_{i_1} \otimes \cdots \otimes v_{i_n} \text{ s.t. } \sum i_j \leq nl \right\rbrace 
\]
and is a sub representation of $\Bn$ (but not a $\Uq$-submodule), see Remark 5.13  from \cite{Jules1}.
Let $\CS^l = ({S^l})^{\otimes n}$, one checks that it is a subspace of $S_{n,l}$. It corresponds to the $\Laurent$-module spanned by the following set:
\[
\CB_{\CS^l} = \left\lbrace v_{i_1} \otimes \cdots \otimes v_{i_n} \in  S_{n,l}  \text{ s.t. } i_k \leq l, \forall k \right\rbrace
\]
and is a sub representation of $\Bn$ from previous remark (and a $\Uq$ submodule). 
%Let $\CS^l = ({S^l})^{\otimes n} \in ({\bfVl})^{\otimes n}$. It is the $\Laurent$-module spanned by the following set:
%\[
%\CB_{\CS^l} = \lbrace v_{i_1} \otimes \cdots \otimes v_{i_n} \in  S_{n,l}  \text{ s.t. } i_k \leq l, \forall k \rbrace
%\]
%and is a sub representation of $\Bn$ from previous remark. We let $\widehat{\CS^l} \in S_{n,nl,{\bfVl}}$ be the $\Laurent$-submodule spanned by:
%\[
%\CB_{\widehat{\CS^l}} = \lbrace v_{i_1} \otimes \cdots \otimes v_{i_n} \in  S_{n,nl,{\bfVl}}  \text{ s.t. } i_k > l, \forall k \rbrace.
%\]
\end{Not}

The following lemma is adapted from \cite[Lemma~1.3]{L-H}.

\begin{lemma}\label{LeHuynhargument}
Let $\beta \in \Bn$ be such that its closure is a knot. Let $Q(\beta)$ be its quantum representation over $\bfVl$ and given by the $\UqhL$ $R$-matrix (see Definition \ref{GoodBnrep}), then:
\[
\Tr(Q(\beta)K^{-1},S_{n,l}) = \Tr(Q(\beta)K^{-1},{\CS^l})
% + \Tr(p_0^{\otimes n} (Q(\beta) \circ K^{-1})  , \widehat{\CS^l}) 
\]
where $\Tr(Q(\beta),Z)$ means the trace of the action $Q(\beta)$ restricted to the $\Bn$-subrepresentation $Z$.
\end{lemma}
\begin{proof}
Let $\beta$ be a braid and $\tau = \perm (\beta)$. The fact that the closure of $\beta$ is a knot guarantees that $\tau$ is an $n$-cycle of $\Sk_n$ (it permutes all the punctures). Let $\left( Q(\beta) K^{-1} \right)_{i_1 , \ldots , i_n}^{s_1,\ldots , s_n}$ be the matrix of $Q(\beta) K^{-1}$ in the basis $\CB(\CS_{nl})$. 
\begin{rmk}\label{LeHuynhsonttropforts}
One has $i_k \leq l$ implies $s_{\tau (k)} \leq l$ for a term not to be $0$. 
\end{rmk}
The latter is due to: $F^{(k)} v_j = 0 $ whenever $k+j > l$, $E v_j = v_{j-1}$ and from the expression of the $R$-matrix and of its inverse that can be found in \cite[Section~1.1.2]{L-H}. These are the same reasons why $S^l$ is stable under the $\Bn$-action. As we want to compute the trace we are only concerned with entries such that $i_k = s_k$. 

Suppose one $i_k$ is less than $l$, then $s_{\tau (k)} \leq l$ (Remark \ref{LeHuynhsonttropforts}), so $i_{\tau (k)} \leq l$ (as we only treat diagonal terms), so that $s_{\tau^2 (k)} \leq l$ (Remark \ref{LeHuynhsonttropforts}) and so on. Finally, all the $i_{\tau^m(k)}$ for $m \in \BN$ must be lower than $l$. As the set $\lbrace \tau^m(1) , 1 \leq m \leq n \rbrace$ is the whole set $\lbrace 1 , \ldots , n \rbrace$, whenever one $i_k$ is lower than $l$ in a vector $v_{i_1} \otimes \cdots \otimes v_{i_n}$, all the $i_k$'s must be lower than $l$ for the corresponding diagonal term not to be $0$. Similarly whenever one $i_k$ is strictly greater than $l$ in a vector $v_{i_1} \otimes \cdots \otimes v_{i_n}$ , all the $i_k$'s must be strictly greater than $l$ for the corresponding diagonal term not to be $0$.

It remains two types of vectors for the diagonal term not to be zero. If all the $i_k$'s are strictly greater than $l$, then $\sum_k i_k \geq n(l+1) > nl$ and $v_{i_1} \otimes \cdots \otimes v_{i_n}$ is not in $S_{nl}$. The only vectors in $S_{nl}$ corresponding to non-zero diagonal terms are the ones such that $i_k \leq l$ for all $k$, namely the ones of $\CS^l$. This concludes the proof. 
%
%This proves the lemma. 
\end{proof}

The finite dimensional representations $\CS^l$ of the braid groups constructed above are the ones used to define the colored Jones polynomial. The following definition (with slightly different conventions, see Remark \ref{framingRemark}) is given at the beginning of Subsection 1.1.4 of \cite{L-H}.

\begin{defn}[Colored Jones polynomial]\label{coloredJonesDef}
Let $n,l \in \BN^*$. If the closure of an $n$-strand braid $\beta$ is the knot $K$, then the $(l+1)$-colored Jones polynomial of $K$ is the following:
\[
\Jones_K(l+1) =  q^{-w(\beta)}\Tr \left( Q(\beta) K^{-1} , \CS^l \right) %q^{w(\beta)\frac{(l+1)^2-1}{2}}
\]
where $w(\beta)$ is the writhe of the braid $\beta$, namely the sum of crossings' signs.  The function $\Tr \left( Q(\beta) K^{-1} , \CS^l \right)$ means the trace of the operator $Q(\beta) K^{-1}$ (see Definition \ref{GoodBnrep} for $Q$, while $K$ is  one of the $\Uq$ generator) restricted to $\CS^l$.
\end{defn}

\begin{rmk}\label{framingRemark}
In \cite{L-H}, the definition has a quadratic factor times the writhe in the exponent of $q$. This quadratic factor is already contain in the definition of our $\Bn$-representations. Namely, in Definition \ref{GoodBnrep}, one observes the multiplication of the $R$-matrix by $q^{-\alpha^2/2}$ explaining this difference in definitions.  
\end{rmk}

\subsection{Colored Jones polynomials compute Lefschetz numbers}\label{subsectionofmaintheorem}

\begin{defn}[Homological trace]
Let $X$ be a topological space and $R$ be a commutative ring such that $H:=H_{\bullet}(X,R)$, the homology associated with $X$ with coefficients in $R$, is a family of free $R$-modules. Let $f$ be a homeomorphism of $X$. The {\em $H$-Lefschetz number} of $f$ is defined to be the following number:
\[
\CL\left( f,H \right) = \sum_i (-1)^i \Tr \left( f_* , H_i \left( X,R \right) \right) \in R
\]
where $\Tr \left( f_* , H_i \left( X,R \right) \right)$ means the trace of the action of $f$ over the $i^{th}$ homology module of $H$. 
\end{defn}

We consider the local system $L_r$ defined over $X_r(w_0 , \ldots , w_n)$ in Definition \ref{localsystXr} under the condition $t=-q^{-2}$.
%, in the unicolored case, namely $\alpha_1 = \ldots = \alpha_n = \alpha$ and with the additional condition that $t=q^{-2}$. 
The coefficients ring is $\Laurent = \BZ \left[ q^{\pm \alpha} , q^{\pm 1} \right]$ or equivalently $\Laurent = \BZ \left[ s , q^{\pm 1} \right]$ (setting $q^{\alpha} = s$). We recall that the braid group $\Bn$ acts by $\Rhom$ over modules $\Hrelm_{\bullet}(X_r) := \Hlf_{\bullet} \left( X_r , X_r^- ; L_r \right)$ from Lemma \ref{RepBn}.

\begin{prop}[A homological trace formula for colored Jones polynomials]\label{traceformulaforColoredJones}
Let $l \in \BN$. If the closure of an $n$-strand braid $\beta$ is the knot $K$, then the $(l+1)$-colored Jones polynomial satisfies the following homological formula:
\begin{align*}
\Jones_K(l+1) =   q^{-w(\beta)-nl}\left( 1+ \sum_{r=1}^{nl} (-1)^r \left[ \CL \left( \Rhom(\beta) , \Hrelm_{\bullet}(X_r) \right) \right]_{\alpha=l}  q^{2r}\right) .
%\sum_{r=0}^{(n-1)l} (-1)^r \left( q^{-nl} \left[ \CL \left( \Rhom(\beta) , \Hrelm_{\bullet}(X_r) \right) \right]_{\alpha_i=l}  - q^{-n(-l-2)} \left[ \CL \left( \Rhom(\beta) , \Hrelm_{\bullet}(X_r) \right) \right]_{\alpha_i=-l-2} \right) q^{2r} \\
%& + q^{-nl} \sum_{r=(n-1)l+1}^{nl} (-1)^r \left[ \CL \left( \Rhom(\beta) , \Hrelm_{\bullet}(X_r) \right) \right]_{\alpha_i=l} q^{2r} . %(q^{w(\beta)\frac{(l+1)^2-1}{2}})
\end{align*}
Traces considered in the sum (namely in the Lefschetz numbers) are elements of $\Laurent$. They are then the specialized according to $\alpha = l$, which corresponds to the ring morphism: 
\[
\aug^l : \bapp
\BZ \left[ s^{\pm 1} , q^{\pm 1} \right] & \to & \BZ \left[ q^{\pm 1} \right] \\
s & \mapsto & q^l 
\eapp
\]
already introduced at the beginning of this section.

\end{prop}
\begin{proof}
From Definition \ref{coloredJonesDef}, the aim of the proof is to compute:
\[
\Tr \left( Q(\beta) K^{-1} , \CS^l \right) .
\]
From Lemma \ref{LeHuynhargument}, we have:
\[
\Tr \left( Q(\beta) K^{-1} , \CS^l \right) = \Tr(Q(\beta)K^{-1},S_{n,l}).
\]
The following lemma computes $\Tr(Q(\beta)K^{-1},S_{nl})$ and concludes the proof.

\begin{lemma}\label{quantumTracetohomologicalTrace}
There is the following formula:
\[
\Tr(Q(\beta)K^{-1} , S_{n,l} ) = q^{-nl}\left( 1 +  \sum_{r=1}^{nl} (-1)^r \left[ \CL \left( \Rhom(\beta) , \Hrelm_{\bullet}(X_r) \right) \right]_{\alpha=l}  q^{2r} \right)
\]
where the traces considered in the sum are elements of the ring $\Laurent$. These elements are then specialized according to $\alpha = l$, corresponding to the ring morphism $\aug^l$.
\end{lemma}
\begin{proof}
One important feature is the following: recall from Proposition \ref{homologystructure} that $\Hrelm_r$ is the only non vanishing module of the complex $\Hlf_{\bullet} \left( X_r , X_r^- ; L_r \right)$. We also know from Theorem \ref{Jules1Theorems} that the $\Laurent$-modules $V_{n,r}$ are $\Bn$-equivalent to $\Hrelm_r$ (for $r \ge 1$). The lemma is then an immediate consequence of:
\[
S_{n,l} = \bigoplus_{r=0}^{nl} V_{n,r}
\]
and the fact that pre-composing by $K^{-1}$ provides a coefficient $q^{-nl} q^{2r}$, corresponding to the action of the diagonal operator $K^{-1}$ over $V_{n,r}$. The first $1$ in the parentheses is the contribution of the trace by the unique vector in $V_{n,0}$. 
\end{proof}

\end{proof}

The above formula suggests that colored Jones polynomials compute the beginning of a generating series of Lefschetz numbers indexed by $r$ over first complexes $\Hlf_{\bullet}(X_r,X_r^-)$. The latter homology complexes are the ones of $X_r$ with local system $L_r$ and relative to $X_r^-$. %The rest of this section is devoted to the interpretation of this homological trace formula in terms of fixed point theory. 

For $r \in \BN^*$, Let $\beta \in \Bn$, $\hat{\beta} \in \Homeo^+(D_n)$, stabilizing $w_0$ and such that the isotopy class of $\hat{\beta}$ is represented by $\beta$. We make the remark that $\hat{\beta}$ can be chosen so to fix $w_0$ as half-Dehn twists corresponding to braid generators can be chosen so to be supported in the interior of the punctured disk (see the description of the action of braid in the sketch of proof of Lemma~\ref{Lawrencerep}. It describes how one associates products of half Dehn twists with braids). Let:
\[
\widehat{\beta}^r : \bapp
(X_r,X_r^-) & \to & (X_r,X_r^-) \\
\lbrace z_1 , \ldots , z_r \rbrace & \mapsto  & \lbrace \widehat{\beta}(z_1) , \ldots , \widehat{\beta}(z_r) \rbrace 
\eapp
\]
be the corresponding self homeomorphism of the pair $(X_r,X_r^-)$. 

\begin{theorem}\label{JonesisLefschetz}
Let $\beta \in \Bn$, and $l \in \BN^*$. Then:
\begin{align*}
\Jones_K(l+1) = q^{-w(\beta)-nl}\left( 1+  \sum_{r=1}^{nl} (-1)^r  \left[ \CL_H \left( \widehat{\beta}^r \right) \right]_{\alpha=l} q^{2r} \right) .
%\\
%& = \sum_{r=0}^{(n-1)l} (-1)^r \left( q^{-nl} \left[ \CL_H \left( \widehat{\beta}^r \right) \right]_{\alpha_i=l}  - q^{-n(-l-2)} \left[ \CL_H \left( \widehat{\beta}^r \right) \right]_{\alpha_i=-l-2} \right) q^{2r} \\
%& + q^{-nl} \sum_{r=(n-1)l+1}^{nl} (-1)^r \left[ \CL_H \left( \widehat{\beta}^r \right) \right]_{\alpha=l} q^{2r} . %(q^{w(\beta)\frac{(l+1)^2-1}{2}}) 
\end{align*}
Here, $\CL_H(\widehat{\beta}^r) \in \BZ H$ (the abelianized Lefschetz number of $\widehat{\beta}^r$) where:
\[
H = \Coker \left( 1- \beta^r_{1,\BZ} \right) \text{ with } \widehat{\beta}^r_{1,\BZ} \in \End(H_1(X_r,X_r^-)) %\right)
\]
the action of $\widehat{\beta}^r$ on $H_1(X_r,X_r^-,\BZ)$. In the formula, $\CL_H(\widehat{\beta}^r)$ is then specialized by the augmentation morphism $\aug^l$. 
\end{theorem}
\begin{proof}
First we need the specialization $\BZ H \to \BZ \left[ q^{\pm 1} \right]$ to be well defined, which is the following lemma.
\begin{lemma}\label{factorsthroughCoker}
Let $\beta \in \Bn$ be a braid and $\widehat{\beta}^r_{1,\BZ}$ its associated action on $H_1(X_r , \BZ)$. Then the augmentation morphism: $\aug^l : \BZ H_1(X_r , \BZ) \to \BZ \left[ q^{\pm 1 } \right]$  given by $\alpha = l$ factors through $\Coker(1-\widehat{\beta}^r_{1,\BZ})$. Namely there exists a morphism $\BZ H \to \BZ \left[ q^{\pm 1} \right]$ such that the following diagram commutes:
\[
\begin{tikzcd}
\BZ H_1 \arrow{r}{\Coker}  \arrow{rd}{\aug^l} 
  & \BZ H \arrow{d}{} \\
    & \BZ \left[ q^{\pm 1} \right]
\end{tikzcd} .
\]
\end{lemma}
\begin{proof}[Proof of Lemma \ref{factorsthroughCoker}]
The proof is immediate from the commutation of the diagram in the proof of Lemma \ref{RepBn}. 
\end{proof}
Then we recall from Proposition \ref{homologystructure} that the complexes $\Hlf_{\bullet} \left( X_r(w_0) ; L_r  \right)$ and $ \Hlf_{\bullet} \left( X_r,X^-_r ; L_r  \right)$ are isomorphic. From the {\em Hopf trace formula} we know that the alternated sum of traces (of a chain map) over a finite dimensional chain complex is computable at the level of homology. Hence, in Proposition \ref{traceformulaforColoredJones}, the Lefschetz numbers involved correspond to an alternated sum of traces of homology actions. They are equal to alternated sum of traces of the corresponding chain complex actions and this fits with Proposition \ref{abelianizedReidemeisterTrace} that introduces the trace formula for abelianized Lefschetz numbers. 
\end{proof}

\begin{rmk}
Let $X_r(w_0) \in X_r$ be the space of configurations with no coordinate in $w_0$. From \cite[Corollary~3.8]{Jules1}, there is the following isomorphism:
\[
\Hlf_r\left( X_r(w_0), \Laurent \right) \simeq \Hrelm_r.
\]
The above homology module on the left is absolute (not relative). It fits with definition of Lefschetz numbers and Nielsen fixed point theory strictly as defined in Section \ref{surveyLefschetz}. Namely, Lefschetz numbers from Theorem \ref{JonesisLefschetz} deal with Nielsen fixed point theory of homeomorphism of $X_r$ fixing $w_0$. 
\end{rmk}

\begin{rmk}
These Lefschetz numbers are the abelianized ones but specialized by the morphism $\aug^l$. They contain the following information.
\begin{itemize}
\item The specialization under $\aug^l$ corresponds to a $mod$ $K$ Lefschetz number (\cite[Chapter~III.2]{J}), such that $K$ is the kernel of $\aug^l \circ \ab : \pi_1(X_r(w_0)) \to \BZ \left[ q^{\pm 1 } \right]$. This kernel is a normal subgroup of $\pi_1(X_r(w_0))$ that corresponds to a lower abelian covering space such that the covering deck transformations group is $\BZ = \BZ \langle q \rangle$. The number of non-zero terms corresponds to the number of the corresponding  $mod$ $K$ Nielsen number that we denote $\CN^{l}(f)$. The latter counts the $mod$ $K$ Nielsen classes, where two fixed points are in the same class iff there exists a path joining them and which differs from its $\widehat{\beta}^r$-image by an element of $K$. 
\item From Remark \ref{lowerboundandlefschetz} we conclude that if one succeed in counting non-zero terms in the weighted expression for the colored Jones polynomials, one would get lower bounds for some classical Nielsen numbers.
\item They contain classical Lefschetz numbers (under the specialization $q \mapsto 1$). 
\end{itemize}
\end{rmk}

\section{Interpretation via homological intersection theory}\label{SectionIntersection}

This Section is devoted to an interpretation of Theorem \ref{JonesisLefschetz} and more precisely Proposition \ref{traceformulaforColoredJones} in terms of homological intersections between precise classes. The colored Jones polynomial of the closure of a braid $\beta$ is defined here as the trace of the quantum representation of $\beta$ over tensor products of Verma modules (more precisely the trace restricted to a subrepresentation). The tensor product of Verma modules is identified with the homology module $\CH := \bigoplus_{r \in \BN} \Hrelm_r$, see Theorem \ref{Jules1Theorems}. This is the main content of the trace formula from Proposition \ref{traceformulaforColoredJones}. 

If $\lbrace e_1, \ldots ,e_n \rbrace$ is a basis of a vector space $V$ and $f \in \End(V)$ then of course:
\[
\Tr(f) = \sum_i \langle f(e_i) , e^i \rangle
\]
where the $e^i$'s refer to the dual basis, and $\langle \cdot , \cdot \rangle$ refers to the corresponding duality pairing (evaluation). In our context while $V$ is the free module $\Hrelm_r$  and basis of $e_i$'s are the code sequences from Definition \ref{codesequences}. In this section, we want to interpret the product $\langle \cdot , \cdot \rangle$ as an intersection pairing between manifolds associated with code sequences and their homologically (yet to be) defined dual.

In Section \ref{SectionPLduality} we state the Poincaré--Lefschetz duality in the context of homology modules $\Hrelm$. In Section \ref{SectionBarcodes} we introduce the family of {\em barcodes} which happens to be dual to the family of code sequences generating modules $\Hrelm$. In Section \ref{SubsectionIntersection} we state and prove Theorem \ref{pairingformulaforJones} expressing colored Jones polynomials in terms of intersection pairings between code sequences and barcodes. We then apply it to the trefoil knot, to give an example of computation.

%\section{Duality interpretation}

\subsection{Poincaré--Lefschetz duality}\label{SectionPLduality}

\begin{lemma}[Poincaré--Lefschetz duality for $\Hrelm_r$]\label{PLdualityrelbound}
The Poincaré--Lefschetz duality in our context provides the following non-degenerate pairing:
\[
\Hrelm_r= \Hlf \left( X_r,X_r^{-}, \Laurent \right) \times \Hnot_r \left( X_r, {X_r^{-}}^{\partial \complement}, \Laurent \right) \to \Laurent. 
\]
with ${X_r^{-}}^{\partial \complement} := \partial X_r  \setminus X_r^{-}$ is the complement of $X_r^{-}$ in the boundary of $X_r$. 
\end{lemma}
\begin{proof}
It is a classical result, see for example \cite[Theorem~3.43]{Hat} for separating the boundary in two (relative) parts as the Lemma suggests. This result in \cite{Hat} requires the following Poincaré duality:
\[
\Hlf_r (X_r; \Laurent) \xrightarrow{\sim} \Hnot^{r}(X_r, \partial X_r ;\Laurent),
\]
where the local ring action is given by the local system defined in \ref{localsystXr}. The latter is the Poincaré duality in the case of non-compact manifolds. One could find such a duality in the case of hyperplane arrangements (in the local system formalism) in \cite{K-1} Relation (2.3), or \cite{AoKi} Lemma 2.9. Or in a slight different formulation after Lemma 4.1 in \cite{Big0} but in the case of this precise unordered configuration space. One remarks that $X_r$ is homotopically equivalent to the quotient of the following hyperplanes complement:
\[
\lbrace (z_1, \ldots , z_r) \in \BC^r \text{ s.t. } z_i \neq z_j \forall i,j, \text{ and } z_i \neq w_k \forall i,k \rbrace
\]
by permutations. 

The pairing is given by intersection number, see \cite[Lemma~2.9]{AoKi}.
\end{proof}

\begin{Not}
For all $r \in \BN$, we fix the following notation:
\[
\Hrelp_r := \Hnot_r \left( X_r, {X_r^{-}}^{\partial \complement}, \Laurent \right) .
\]
and:
\[
X_r^+:= {X_r^{-}}^{\partial \complement} 
\]
\end{Not}

While $\Hrelm_r= \Hlf \left( X_r,X_r^{-}, \Laurent \right)$ is generated by code sequences, it remains to describe its dual basis regarding the above pairing. It is the purpose of the next section.

\subsection{Barcodes dual basis}\label{SectionBarcodes}

\begin{defn}[Barcodes]\label{barcodes}
For $r \in \BN$, and $(k_0,\ldots, k_{n-1})$ integers such that $\sum k_i =r$.  A {\em barcode} $B(k_0,\ldots, k_{n-1})$ is the following diagram:

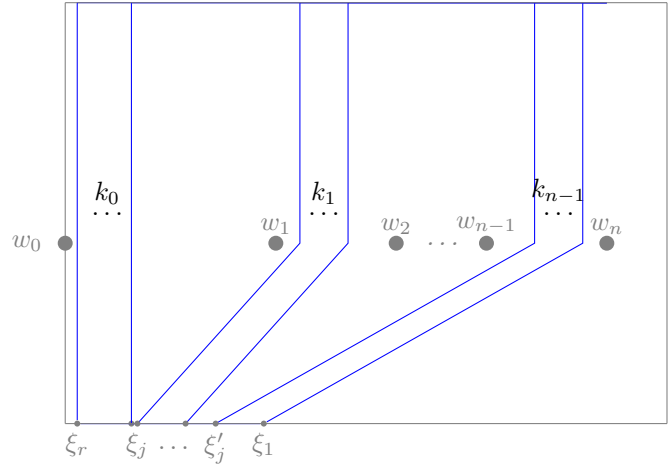
\begin{equation*}
\begin{tikzpicture}[scale=0.8,decoration={
    markings,
    mark=at position 0.5 with {\arrow{<}}}
    ] 
\node (w0) at (-5,0) {};
\node (w00) at (5,0) {};
\node (w1) at (-1.5,0) {};
\node (w2) at (0.5,0) {};
\node[gray] at (1.3,0.0) {\ldots};
\node (wn) at (4,0) {};
\node (wn1) at (2,0) {};

\node[gray,circle,fill,inner sep=0.8pt] (xir) at (-4.8,-3) {};
\node[below,gray] at (xir) {$\xi_r$};
\node[below=5pt,gray] at (-3.2,-3) {$\ldots$};
\node[gray,circle,fill,inner sep=0.8pt] (xi1) at (-1.7,-3) {};
\node[below,gray] at (xi1) {$\xi_1$};
\node[gray,circle,fill,inner sep=0.8pt] (xirk0) at (-3.9,-3) {};
%\node[below,gray] at (xirk0) {$\xi_{r-k_0}$};

\coordinate (a) at (-3,-3);
\coordinate (b) at (4,4);

%\node[gray] at (w00)[above] {$w_{\infty}$};
\node[gray] at (wn)[above] {$w_n$};
\node[gray] at (wn1)[above] {$w_{n-1}$};
\node[gray] at (w2)[above] {$w_2$};
\node[gray] at (w1)[above] {$w_1$};
\node[gray] at (w0) [left=5pt] {$w_0$};
\foreach \n in {w1,w2,wn,wn1}
  \node at (\n)[gray,circle,fill,inner sep=2pt]{};
\node at (w0)[gray,circle,fill,inner sep=2pt]{};

\draw[blue] (xi1)--(xir)|-(b);
\node at (-4.3,0.5) {$\ldots$};
\node[above] at (-4.3,0.5) {$k_0$};
\draw[blue] (-3.9,-3)--(-3.9,-3)|-(b);

\draw[blue] (-3.8,-3)--(-1.1,0)--(-1.1,0)|-(b);
\node at (-0.7,0.5) {$\ldots$};
\node[above] at (-0.7,0.5) {$k_1$};
\draw[blue] (-3,-3)--(-0.3,0)--(-0.3,0)|-(b);

\draw[blue] (-2.5,-3)--(2.8,0)--(2.8,0)|-(b);
\node at (3.2,0.5) {$\ldots$};
\node[above] at (3.2,0.5) {$k_{n-1}$};
\draw[blue] (xi1)--(3.6,0)--(3.6,0)|-(b);

\node[gray,circle,fill,inner sep=0.8pt] (xij1) at (-3.8,-3) {};
\node[gray,circle,fill,inner sep=0.8pt] (xij2) at (-2.5,-3) {};
\node[gray,circle,fill,inner sep=0.8pt] (xij3) at (-3,-3) {};
\node[below,gray] at (xij1) {$\xi_j$};
\node[below,gray] at (xij2) {$\xi'_j$};

%\draw[gray] (-5,4) -- (-5,1);
\draw[gray] (-5,4) -- (-5,-3);
\draw[gray] (-5,-3) -- (5,-3);% node[right] {$\partial D_n$};
\draw[gray] (5,4) -- (5,-3);
\draw[gray] (5,4) -- (-5,4);
\end{tikzpicture} .
\end{equation*}
where $\xi_j=\xi_{r-(k_0+1)}$ and $\xi_{j'}=\xi_{k_{n-1}}$ are defined only for informal reasons. It consists in $r$ arcs attached (at the bottom) to $\lbrace \xi_1 , \ldots, \xi_r \rbrace$. The $k_0$ leftmost are passing between $w_0$ and $w_1$, next $k_1$ between $w_1$ and $w_2$ and so on, until rightmost $k_{n-1}$ arcs passing between $w_{n-1}$ and $w_n$. They naturally define a class in  $\Hrelp_r$ as they are {\em noodles} (defined in \cite{Big1}). We recall briefly the protocole. 
Let $I$ be an interval, one can associate with the union of blue arcs the embedding:
\[
I^r \to X_r
\]
where copies of $I$ are sent (one by one) to blue arcs from the leftmost to the rightmost. As the boundary of such a hypercube lies in the relative part ${X_r^{+}}$, this embedding defines a class in $\Hnot(X_r,X_r^+;\BZ)$. There is only one lift of this hypercube containing $\widehat{{\pmb \xi}^r}$ (the chosen lift of base point) in $\widehat{X_r}$. It allows to choose a lift $\widehat{B}(k_0,\ldots,k_{n-1})$ of the hypercube to $\widehat{X_r}$.  It defines a class in $\Hrelp_+$ that we still call barcode by extension and that we denote $\left[ B(k_0,\ldots, k_{n-1}) \right]$. 
\end{defn}

Now we describe the intersection pairing arising from Lemma \ref{PLdualityrelbound}. Computations for the case of hyperplane arrangements (that applies in our context) and in terms of intersection number with local coefficients are made for example in \cite[Section~2.3.3]{AoKi}, but with quite different notations for the local system formalism. Here, we try to recover a construction adapted with our notations, following S. Bigelow's work (see for instance \cite{Big1,Big2}). Let $M_{\pm}$ be manifold representatives of cycles of the pairs $(X_r,X_r^+)$ and $(X_r,X_r^-)$ respectively. 
\begin{itemize}
\item $M_+$ and $M_-$ intersect in a finite number of point. (Here the locally finite hypothesis on the left is important)
\item By considering cycles of $\Hrelp_r$ and $\Hrelm_r$ respectively, it defines lifts $\widehat{M_{\pm}}$ of $M_{\pm}$ to $\widehat{X_r}$. 
\item Let $p \in M_+ \cap M_-$, and $\widehat{p}$ a lift of $p$ in $\widehat{M_-}$.
\item By the unique lift property for covering spaces, there exists a unique $m_p \in \Laurent$ such that $\widehat{M_-}$ and $m_p \widehat{M_+}$ intersect in $\widehat{p}$. 
\item The intersection between these two manifolds is a sign $\epsilon_p \in \lbrace \pm 1 \rbrace$.
\item By the $\Laurent$-linearity, at point $p$, the contribution to the intersection between $M_-$ and $M_+$ with local coefficients is $\epsilon_p m_p$. 
\item Finally:
\[
\langle M_-,M_+ \rangle_\Laurent = \sum_{p \in M_- \cap M_+} \epsilon_p m_p .
\]
%\item This construction does not depend on the choice of lifts for $M_{\pm}$ thanks to the $\Laurent$-linearity. 
\end{itemize}

We do a concrete example of computation for standard code sequences and bar codes in the case $n=1$.

\begin{example}[$n=1$]\label{THEexample}
In this example we work with only one puncture, in the case $n=1$. For $r \in \BN$ we want to compute:
\[
\langle U(r) , B(r) \rangle_{\Laurent}
\]
the dual pairing between standard code sequence and bar code. Here is the corresponding diagram. 

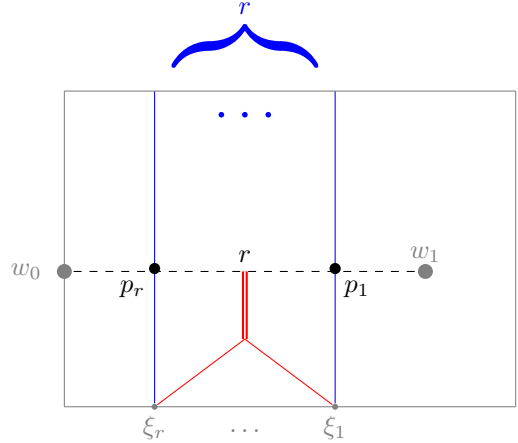
\begin{equation*}
\begin{tikzpicture}[scale=0.6,decoration={
    markings,
    mark=at position 0.5 with {\arrow{<}}}
    ] 
\node (w0) at (-5,0) {};
\node (w00) at (5,0) {};
\node (w1) at (3,0) {};
%\node (w2) at (0.5,0) {};
%\node[gray] at (1.3,0.0) {\ldots};
%\node (wn) at (4,0) {};
%\node (wn1) at (2,0) {};

\node[gray,circle,fill,inner sep=0.8pt] (xir) at (-3,-3) {};
\node[below,gray] at (xir) {$\xi_r$};
\node[below=5pt,gray] at (-1,-3) {$\ldots$};
\node[gray,circle,fill,inner sep=0.8pt] (xi1) at (1,-3) {};
\node[below,gray] at (xi1) {$\xi_1$};
%\node[gray,circle,fill,inner sep=0.8pt] (xirk0) at (-3.9,-3) {};
%\node[below,gray] at (xirk0) {$\xi_{r-k_0}$};

\coordinate (a) at (-3,-3);
\coordinate (b) at (4,4);

%\node[gray] at (w00)[above] {$w_{\infty}$};
%\node[gray] at (wn)[above] {$w_n$};
%\node[gray] at (wn1)[above] {$w_{n-1}$};
%\node[gray] at (w2)[above] {$w_2$};
\node[gray] at (w1)[above] {$w_1$};
\node[gray] at (w0) [left=5pt] {$w_0$};
\foreach \n in {w1}
  \node at (\n)[gray,circle,fill,inner sep=2pt]{};
\node at (w0)[gray,circle,fill,inner sep=2pt]{};

%\node[blue,thin,scale=2] (-1,5) {$\overbrace{}$};

\draw[blue] (xir)--(xir|-b);
%\node at (-4.3,1.5) {$\ldots$};
\node[blue,above] at (-1,5.5) {$r$};
\node[blue,above,scale=2] at (-1,3) {$\ldots$};
\node[blue,scale=3,thin] at (-1,5) {$\overbrace{}$};
\draw[blue] (xi1)--(xi1|-b);

%\node[gray,circle,fill,inner sep=0.8pt] (xij1) at (-3.8,-3) {};
%\node[gray,circle,fill,inner sep=0.8pt] (xij2) at (-2.5,-3) {};
%\node[gray,circle,fill,inner sep=0.8pt] (xij3) at (-3,-3) {};
%\node[below,gray] at (xij1) {$\xi_j$};
%\node[below,gray] at (xij2) {$\xi'_j$};

\draw[dashed] (w0) to node[midway,above] (k0) {$r$} (w1);

\coordinate (c) at (-1.5,-1.5);

\draw[double,red,thick] (k0)--(k0|-c);
\draw[red] (k0|-c)--(xi1);
\draw[red] (k0|-c)--(xir);

\node[scale=4] at (xi1|-w0) {$\cdot$};
\node[below right] at (xi1|-w0) {$p_1$};
\node[scale=4] at (xir|-w0) {$\cdot$};
\node[below left] at (xir|-w0) {$p_r$};

%\draw[gray] (-5,4) -- (-5,1);
\draw[gray] (-5,4) -- (-5,-3);
\draw[gray] (-5,-3) -- (5,-3);% node[right] {$\partial D_n$};
\draw[gray] (5,4) -- (5,-3);
\draw[gray] (5,4) -- (-5,4);
\end{tikzpicture} .
\end{equation*}
In this diagram one sees where the (indexed $r$) dashed arc from the multi-arcs intersects the several plain (blue) arcs from the barcode in $D_n$. Namely there are $r$ intersection points denoted $p_1,\ldots, p_r$ from right to left. There is a unique intersection point in $U(r) \cap B(r)$ (considered as subspaces of $X_r$), namely the configuration ${\pmb p} := \lbrace p_1 , \ldots , p_r \rbrace$. This guarantees, following notations described above, that:
\[
\langle U(r) , B(r) \rangle_{\Laurent} = \epsilon_{\pmb p} m_{\pmb p} .
\]
The coefficient $m_{\pmb p }$ can be computed by considering the path $\delta_{\pmb p}$ of $X_r$ constructed as the composition of the following steps:
\begin{itemize}
\item First the path going from $\lbrace \xi_1 , \ldots , \xi_r \rbrace$ to $U(r)$ following red handles,
\item then joining $\lbrace p_1 , \ldots , p_r \rbrace$ going along $U(r)$,
\item then going back to ${\pmb \xi}^r$ running along $B(r)$.
\end{itemize}
This path is a loop of $X_r$ based on ${\pmb \xi}$. By considering one of its lift to $\widehat{X_r}$, one can check that it relates $\widehat{\pmb \xi}$ and $m_{\pmb p } \widehat{\pmb \xi}$, where $\widehat{\pmb \xi}$ is a lift of the base point ${\pmb \xi}$. The explanation of this fact is exactly the same as the one given before Lemma 3.11 in \cite{Jules2} which is adapted from Section 3.1 of \cite{Big1}.  Knowing this, we directly conclude:
\[
m_{\pmb p} = \rho_r(\delta_{\pmb p }),
\]
and moreover that:
\[
m_{\pmb p} = \rho_r(\delta_{\pmb p })=1. 
\]
One can check that the braid given by $\delta_{\pmb p}$ seen as an element of $\pi_1(X_r,{\pmb \xi})$ (following the model \cite[Remark~2.2]{Jules1}) is trivial. This ensures the above equality by the definition of $\rho_r$ (Definition \ref{localsystXr}). 

It remains to compute $\epsilon_{\pmb p}$. The intersection at this point is positive, one can check \cite[Claim~3.3]{Big1} for an explanation (orientation between dashed arc and blue arcs is positive, then the monomial $m_{\pmb p}$ has a pair exponent of $t$).
Finally, for all $r \in \BN$:
\begin{equation}\label{ArLrisone}
\langle U(r) , B(r) \rangle_{\Laurent} = 1.
\end{equation}

\end{example}

\begin{rmk}[Homological interpretation for quantum factorials]
We re-compute the pairing from Example \ref{THEexample} to see quantum factorials appearing from the pairing. We recall the following expression relating a standard code sequence with a {\em standard multi fork} (defined in \cite[Section~3.1]{Jules1}) happening in $\Hrelm_r$:
\begin{equation}\label{plaintodashed}
\vcenter{\hbox{
\begin{tikzpicture}[decoration={
    markings,
    mark=at position 0.5 with {\arrow{>}}}
    ]
\node (w0) at (-1,0) {};
\node (w1) at (1,0) {};
\coordinate (a) at (-1,-1);

\draw[postaction={decorate}] (w0) to[bend left=40] node[pos=0.3,above] (k1) {} (w1);
\node at (w1)[above=4pt,left=23pt] {.};
\node at (w1)[left=23pt] {.};
\node at (w1)[below=4pt,left=23pt] {.};
\node at (w1)[above=20pt,left=20pt] {$k$};
\draw[postaction={decorate}] (w0) to[bend left=-40] node[pos=0.7,above] (kk) {} (w1);

\node[gray] at (w0)[left=5pt] {$w_i$};
\node[gray] at (w1)[right=5pt] {$w_j$};
\foreach \n in {w0,w1}
  \node at (\n)[gray,circle,fill,inner sep=3pt]{};

\draw[red] (k1) -- (k1|-a);
\draw[red] (kk) -- (kk|-a);

\end{tikzpicture}
}}
=
(k)_{-t}! \vcenter{\hbox{
\begin{tikzpicture}[decoration={
    markings,
    mark=at position 0.5 with {\arrow{>}}}
    ]
\node (w0) at (-1.5,0) {};
\node (w1) at (1.5,0) {};
\coordinate (a) at (-1,-1);

\draw[dashed] (w0) -- (w1) node[pos=0.5,above] (k0) {$k$};

\node[gray] at (w0)[left=5pt] {$w_i$};
\node[gray] at (w1)[right=5pt] {$w_j$};
\foreach \n in {w0,w1}
  \node at (\n)[gray,circle,fill,inner sep=3pt]{};

\draw[double,red,thick] (k0) -- (k0|-a);

\end{tikzpicture}
}}.
\end{equation}
It arises straightforwardly from \cite[Corollary~4.10]{Jules1}. The diagram of the multi-fork on the left corresponds to a lift of a hypercube similarly as for the barcodes from Definition \ref{barcodes} above, see \cite[Section~3.1]{Jules1} for the detailed construction. We denote it $F(r)$. We start by computing the pairing $\langle F(r),B(r) \rangle_{\Laurent}$. Here is the corresponding diagram:
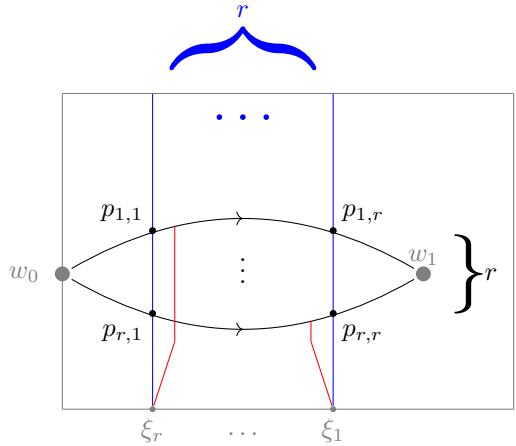
\begin{equation*}
\begin{tikzpicture}[scale=0.6,decoration={
    markings,
    mark=at position 0.5 with {\arrow{>}}}
    ] 
\node (w0) at (-5,0) {};
\node (w00) at (5,0) {};
\node (w1) at (3,0) {};
%\node (w2) at (0.5,0) {};
%\node[gray] at (1.3,0.0) {\ldots};
%\node (wn) at (4,0) {};
%\node (wn1) at (2,0) {};

\node[gray,circle,fill,inner sep=0.8pt] (xir) at (-3,-3) {};
\node[below,gray] at (xir) {$\xi_r$};
\node[below=5pt,gray] at (-1,-3) {$\ldots$};
\node[gray,circle,fill,inner sep=0.8pt] (xi1) at (1,-3) {};
\node[below,gray] at (xi1) {$\xi_1$};
%\node[gray,circle,fill,inner sep=0.8pt] (xirk0) at (-3.9,-3) {};
%\node[below,gray] at (xirk0) {$\xi_{r-k_0}$};

\coordinate (a) at (-3,-3);
\coordinate (b) at (4,4);

%\node[gray] at (w00)[above] {$w_{\infty}$};
%\node[gray] at (wn)[above] {$w_n$};
%\node[gray] at (wn1)[above] {$w_{n-1}$};
%\node[gray] at (w2)[above] {$w_2$};
\node[gray] at (w1)[above] {$w_1$};
\node[gray] at (w0) [left=5pt] {$w_0$};
\foreach \n in {w1}
  \node at (\n)[gray,circle,fill,inner sep=2pt]{};
\node at (w0)[gray,circle,fill,inner sep=2pt]{};

%\node[blue,thin,scale=2] (-1,5) {$\overbrace{}$};

\draw[blue] (xir)--(xir|-b);
%\node at (-4.3,1.5) {$\ldots$};
\node[blue,above] at (-1,5.5) {$r$};
\node[blue,above,scale=2] at (-1,3) {$\ldots$};
\node[blue,scale=3,thin] at (-1,5) {$\overbrace{}$};
\draw[blue] (xi1)--(xi1|-b);

\node[scale=3,thin] at (4,0) {$\rbrace$};
\node at (4.5,0) {$r$};
\node at (-1,0.25) {$\vdots$};

%\node[gray,circle,fill,inner sep=0.8pt] (xij1) at (-3.8,-3) {};
%\node[gray,circle,fill,inner sep=0.8pt] (xij2) at (-2.5,-3) {};
%\node[gray,circle,fill,inner sep=0.8pt] (xij3) at (-3,-3) {};
%\node[below,gray] at (xij1) {$\xi_j$};
%\node[below,gray] at (xij2) {$\xi'_j$};

\draw[postaction={decorate}] (w0) to[bend right] node[pos=0.7,above] (k0) {} node[pos=0.235,scale=2.5] (pr1) {$\cdot$} node[pos=0.765,scale=2.5] (prr) {$\cdot$} (w1);
\draw[postaction={decorate}] (w0) to[bend left] node[pos=0.3,above] (k1) {} node[pos=0.235,scale=2.5] (p11) {$\cdot$} node[pos=0.765,scale=2.5] (p1r) {$\cdot$}  (w1);

\node[above left] at (p11) {$p_{1,1}$};
\node[above right] at (p1r) {$p_{1,r}$};
\node[below left] at (pr1) {$p_{r,1}$};
\node[below right] at (prr) {$p_{r,r}$};

\coordinate (c) at (-1.5,-1.5);

\draw[red] (k0)--(k0|-c);
\draw[red] (k0|-c)--(xi1);
\draw[red] (k1)--(k1|-c);
\draw[red] (k1|-c)--(xir);

%\draw[gray] (-5,4) -- (-5,1);
\draw[gray] (-5,4) -- (-5,-3);
\draw[gray] (-5,-3) -- (5,-3);% node[right] {$\partial D_n$};
\draw[gray] (5,4) -- (5,-3);
\draw[gray] (5,4) -- (-5,4);
\end{tikzpicture} .
\end{equation*}

This time there are $r$ arcs from the multi-fork intersecting $r$ arcs from the barcode, so that there are $r^2$ intersection point in $D_n$. They fit on a grid so that we denote them:
\[
\begin{pmatrix}
p_{1,1} & \cdots & p_{1,r} \\
\vdots & \vdots & \vdots \\
p_{r,1} & \cdots & p_{r,r}
\end{pmatrix} . 
\]
To get an intersection configuration in $X_r$, one has to choose one and only one point for each line and column. Such a choice corresponds to a choice of an $r\times r$ permutation matrix. For example to the identity matrix corresponds the configuration:
\[
{\pmb p}_{\Id} = \lbrace p_{1,1}, p_{2,2} , \ldots , p_{r,r} \rbrace.
\]
From the description of the pairing, one can check that the intersection is positive at that point. To compute the corresponding monomial $m_{{\pmb p}_{\Id}} \in \Laurent$, one should consider this time the path $\delta_{{\pmb p}_{\Id}}$ of $X_r$ constructed as the composition of the following steps:
\begin{itemize}
\item First the path going from $\lbrace \xi_1 , \ldots , \xi_r \rbrace$ to $F(r)$ following red handles,
\item then joining $\lbrace p_1 , \ldots , p_r \rbrace$ going along $F(r)$,
\item finally going back to $\lbrace \xi_1 , \ldots , \xi_r \rbrace$ following $B(r)$.
%\item finally going back to $\lbrace \xi_1 , \ldots , \xi_r \rbrace$ following the (opposite of) green handles. 
\end{itemize}
Again this path is a loop, its image by $\rho_r$ gives $m_{{\pmb p}_{\Id}}$. While (as the previous example) this loop is homotopically trivial, $m_{{\pmb p}_{\Id}} = 1$. If one chooses another permutation matrix, then the first step for the construction of the loop $\delta_{\pmb p}$ could be different (together with the sign of the intersection). Namely by changing red handles, the image of $\delta_{\pmb p}$ by $\rho_r$ could be changed by a power of $t$. Let $P$ be a permutation matrix, and ${\pmb p}_{P}$ the corresponding intersection configuration. For computing the monomial $m_{{\pmb p}_{P}}$ one has to perform some ``knitting'' operation on red handles to pass from the trivial loop to the one corresponding to the permutation $P$. Such an operation would correspond to applying the braid associated with the minimal decomposition in transposition of $P$. If $k$ is the number of transposition involved, then one can check $m_{{\pmb p}_{P}} = t^k$, and the sign of the intersection is $(-1)^k$. Finally:
\begin{equation}\label{transposum}
\langle F(r),B(r) \rangle_{\Laurent} = \sum_{k< \frac{r(r-1)}{2}} a_k (-t)^k
\end{equation}
where $a_k$ is the number of permutation having their minimal transposition decomposition involving $k$ transpositions. 

Considering this, one recovers:
\begin{equation}
\sum_{k< \frac{r(r-1)}{2}} a_k (-t)^k = \langle F(r),B(r) \rangle_{\Laurent} = (r)_{-t}! \langle A(r),L(r) \rangle_{\Laurent} = (r)_{-t}! .
\end{equation}
First equality is the above Relation (\ref{transposum}), the second is Relation (\ref{plaintodashed}), and the last one is Example \ref{THEexample} above.

\end{rmk}

The above example shows the following proposition, stating that code sequences and barcodes constitute dual bases regarding the Poincaré dual pairing.

\begin{prop}[Dual bases]\label{Dualbases}
Let $r \in \BN$ and $\CB := \lbrace B(k_0 , \ldots , k_{n-1})  \rbrace_{\sum k_i = r}$ the family of barcodes. The family $\CB$ of elements in $\Hrelp_r$ is the dual basis for $\CU$ regarding the intersection pairing $\langle \cdot, \cdot \rangle_{\Laurent}$. More precisely:
\[
\langle U(k_0 , \ldots , k_{n-1} ) , B(k'_0 , \ldots , k'_{n-1} \rangle_{\Laurent} = \delta_{(k_0 , \ldots k_{n-1}),(k'_0,\ldots , k'_{n-1})},
\]
where $\delta$ is a (list) Kronecker symbol. 
\end{prop}
\begin{proof}
For ${\bf k} := (k_0 , \ldots, k_{n-1})$ (resp. ${\bf k'} := (k'_0 , \ldots , k'_{n-1})$) such that $\sum k_i =r$ (resp. $\sum k'_i=r$), the picture illustrating the intersection of the proposition is the following:
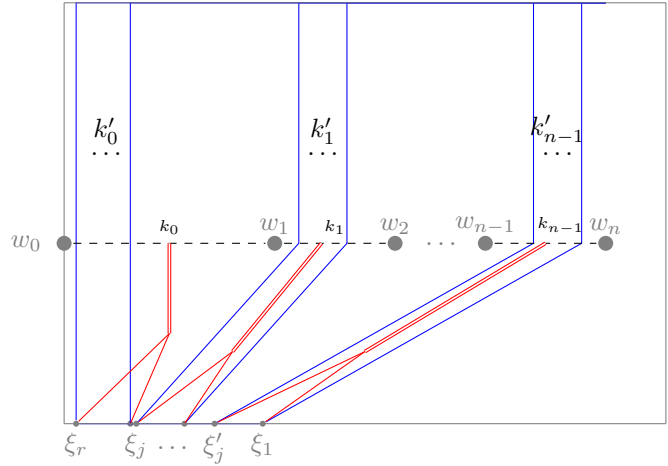
\begin{equation*}
\begin{tikzpicture}[scale=0.8,decoration={
    markings,
    mark=at position 0.5 with {\arrow{<}}}
    ] 
\node (w0) at (-5,0) {};
\node (w00) at (5,0) {};
\node (w1) at (-1.5,0) {};
\node (w2) at (0.5,0) {};
\node[gray] at (1.3,0.0) {\ldots};
\node (wn) at (4,0) {};
\node (wn1) at (2,0) {};

\node[gray,circle,fill,inner sep=0.8pt] (xir) at (-4.8,-3) {};
\node[below,gray] at (xir) {$\xi_r$};
\node[below=5pt,gray] at (-3.2,-3) {$\ldots$};
\node[gray,circle,fill,inner sep=0.8pt] (xi1) at (-1.7,-3) {};
\node[below,gray] at (xi1) {$\xi_1$};
\node[gray,circle,fill,inner sep=0.8pt] (xirk0) at (-3.9,-3) {};
%\node[below,gray] at (xirk0) {$\xi_{r-k_0}$};

\coordinate (a) at (-3,-3);
\coordinate (b) at (4,4);

%\node[gray] at (w00)[above] {$w_{\infty}$};
\node[gray] at (wn)[above] {$w_n$};
\node[gray] at (wn1)[above] {$w_{n-1}$};
\node[gray] at (w2)[above] {$w_2$};
\node[gray] at (w1)[above] {$w_1$};
\node[gray] at (w0) [left=5pt] {$w_0$};
\foreach \n in {w1,w2,wn,wn1}
  \node at (\n)[gray,circle,fill,inner sep=2pt]{};
\node at (w0)[gray,circle,fill,inner sep=2pt]{};

\draw[blue] (xi1)--(xir)|-(b);
\node at (-4.3,1.5) {$\ldots$};
\node[above] at (-4.3,1.5) {$k'_0$};
\draw[blue] (-3.9,-3)--(-3.9,-3)|-(b);

\draw[blue] (-3.8,-3)--(-1.1,0)--(-1.1,0)|-(b);
\node at (-0.7,1.5) {$\ldots$};
\node[above] at (-0.7,1.5) {$k'_1$};
\draw[blue] (-3,-3)--(-0.3,0)--(-0.3,0)|-(b);

\draw[blue] (-2.5,-3)--(2.8,0)--(2.8,0)|-(b);
\node at (3.2,1.5) {$\ldots$};
\node[above] at (3.2,1.5) {$k'_{n-1}$};
\draw[blue] (xi1)--(3.6,0)--(3.6,0)|-(b);

\node[gray,circle,fill,inner sep=0.8pt] (xij1) at (-3.8,-3) {};
\node[gray,circle,fill,inner sep=0.8pt] (xij2) at (-2.5,-3) {};
\node[gray,circle,fill,inner sep=0.8pt] (xij3) at (-3,-3) {};
\node[below,gray] at (xij1) {$\xi_j$};
\node[below,gray] at (xij2) {$\xi'_j$};

\draw[dashed] (w0) to node[midway,above] (k0) {\tiny $k_0$} (w1);
\draw[dashed] (w1) to node[midway,above] (k1) {\tiny $k_1$} (w2);
\draw[dashed] (wn1) to node[pos=0.65,above] (kn1) {\tiny $k_{n-1}$} (wn);

\coordinate (c) at (-1.5,-1.5);

\draw[double,red] (k0)--(k0|-c);
\draw[red] (k0|-c) -- (xir);
\draw[red] (k0|-c) -- (-3.9,-3);

\draw[double,red] (k1)--(-2.2,-1.8);
\draw[red] (-2.2,-1.8) -- (xij1);
\draw[red] (-2.2,-1.8) -- (-3,-3);

\draw[double,red] (3,0)--(0,-1.8);
\draw[red] (0,-1.8) -- (xij2);
\draw[red] (0,-1.8) -- (xi1);

%\draw[gray] (-5,4) -- (-5,1);
\draw[gray] (-5,4) -- (-5,-3);
\draw[gray] (-5,-3) -- (5,-3);% node[right] {$\partial D_n$};
\draw[gray] (5,4) -- (5,-3);
\draw[gray] (5,4) -- (-5,4);
\end{tikzpicture} .
\end{equation*}

More precisely, code sequences and barcodes can be put in relative position such as that of the above figure up to isotopy (not modifying the pairing). In this position, there are $k'_0+\cdots + k'_{n-1}=r$ intersection points between (the $n$) dashed arcs and (the $r$) (blue) plain arcs. For a configuration in $A({\bf k})$ there are $k_0$ embedded points in the indexed $k_0$ dashed arc, $k_1$ in the indexed $k_1$ and so on. This ensures that to obtain a configuration of $r$ points belonging to both $U({\bf k})$ and $B({\bf k'})$ the following indexes must be equal:
\[
k_0 = k'_0 , \ldots , k_{n-1} = k'_{n-1}.
\]
If one of these equality is not true, the intersection is empty and the pairing is zero.

Now suppose we are in the case where ${\bf k}={\bf k'}$. Then there is a unique intersection point between $U({\bf k})$ and $B({\bf k})$ corresponding to the configuration of intersection points between plain arcs and dashed arcs in the above diagram. The computation is the same as in Example \ref{THEexample} so to obtain $1$.
\end{proof}

\subsection{Application to colored Jones polynomials}\label{SubsectionIntersection}

\subsubsection{The Formula}

\begin{Not}
Let:
\[
\langle \cdot, \cdot \rangle_{\aug^l} := \aug \circ \left( \langle \cdot,\cdot \rangle_{\Laurent} \right) : \Hrelm_r \times \Hrelp_r \to \Laurent_0:=\BZ \left[ q^{\pm 1} \right]
\]
be the intersection pairing from the Poincaré--Lefschetz duality with coefficients in $\Laurent$ composed with the augmentation morphism specializing $\alpha$ to $l \in \BN$. 
\end{Not}

\begin{thm}\label{pairingformulaforJones}
Let $\beta \in \Bn$ a braid such that its closure is the knot $K$. Then, the colored Jones polynomials satisfy the following formulae:
\begin{align}
\Jones_K(l+1) & =   q^{-w(\beta)-nl}\left( 1+ \sum_{r=1}^{nl} \sum_{\begin{array}{c} (k_0,\ldots,k_{n-1})/ \\ \sum k_i=r \end{array}} \left\langle \Rhom\left(\beta \right) \cdot U(k_0,\ldots, k_{n-1}) , B(k_0,\ldots, k_{n-1}) \right\rangle_{\aug^l} q^{2r} \right) \\
& = q^{-w(\beta)-nl} \left( 1+  \sum_{\begin{array}{c} (k_0,\ldots,k_{n-1}) / \\ \sum k_i\le nl \end{array}} \left\langle \Rhom\left(\beta \right)  \cdot U(k_0,\ldots, k_{n-1}) , B(k_0,\ldots, k_{n-1}) \right\rangle_{\aug^l} q^{2\left(\sum_i k_i\right)} \right) ,
\end{align}
where $\Rhom\left(\beta \right)  \cdot U(k_0,\ldots , k_{n-1})$ refers to the homological (Lawrence) action of $\beta$ upon the class $U(k_0,\ldots , k_{n-1})$. 
\end{thm}
\begin{proof}
It is a direct consequence of Proposition \ref{traceformulaforColoredJones}. The trace is interpreted as the dual pairing arising from the Poincaré--Lefschetz duality applied with dual bases of code sequences and barcodes. More precisely, from Proposition \ref{Dualbases}, if $f \in \End_{\Laurent}\left( \Hrelm_r \right)$, then:
\[
\Tr(f) = \sum_{\begin{array}{c} (k_0,\ldots,k_{n-1})/ \\ \sum k_i=r \end{array}} \left\langle f \cdot U(k_0,\ldots, k_{n-1}) , B(k_0,\ldots, k_{n-1}) \rangle_{\Laurent}\right.
\]
\end{proof}

\begin{rmk}
For a practical use of the above theorem, one can choose the expression of $\beta$ in terms of product of half Dehn twists, and make it act by homeomorphism on the manifold associated with $U(k_0,\ldots , k_{n-1})$ before computing its intersection number with the corresponding barcode. As the pairing is invariant under isotopies, this number does not depend on the choice of homeomorphism representative for $\beta$. 
\end{rmk}

The formula from the above theorem should recover formulae from \cite{Big3,Ito2,An}.

\subsubsection{The trefoil}

In this section we do the computation of the Jones polynomial of the trefoil knot. The trefoil knot is the closure of the braid $\sigma_1^3 \in \CB_2$, where $\sigma_1$ is the only (half Dehn twist) generator of the braids in two strands. For $(k_0,k_1) \in \BN^2$, such that $k_1+k_2=r\in \BN$, we have:
\begin{equation*}
\sigma_1^3 \left( \vcenter{\hbox{
\begin{tikzpicture}[scale=0.5,decoration={
    markings,
    mark=at position 0.5 with {\arrow{<}}}
    ] 
\node (w0) at (-5,0) {};
\node (w00) at (5,0) {};
\node (w1) at (-2,0) {};
\node (w2) at (2,0) {};
%\node[gray] at (1.3,0.0) {\ldots};
%\node (wn) at (4,0) {};
%\node (wn1) at (2,0) {};

\node[gray,circle,fill,inner sep=0.8pt] (xir) at (-4.8,-3) {};
\node[below,gray] at (xir) {$\xi_r$};
%\node[below=5pt,gray] at (-3.2,-3) {$\ldots$};
\node[gray,circle,fill,inner sep=0.8pt] (xi1) at (-2.3,-3) {};
\node[below,gray] at (xi1) {$\xi_1$};
\node[gray,circle,fill,inner sep=0.8pt] (xirk0) at (-3.9,-3) {};
%\node[below,gray] at (xirk0) {$\xi_{r-k_0}$};
\node[gray,circle,fill,inner sep=0.8pt] (xik1) at (-3.6,-3) {};
\node[below,gray] at (xik1) {$\xi_{k_1}$};

\coordinate (a) at (-3,-3);
\coordinate (b) at (4,4);

\draw[dashed] (w0) to node[pos=0.3,above] (k0) {$k_0$} (w1) ;
\draw[dashed] (w1) to node[pos=0.3,above] (k1) {$k_1$} (w2) ;

%\node[gray] at (w00)[above] {$w_{\infty}$};
%\node[gray] at (wn)[above] {$w_n$};
%\node[gray] at (wn1)[above] {$w_{n-1}$};
\node[gray] at (w2)[above] {$w_2$};
\node[gray] at (w1)[above] {$w_1$};
\node[gray] at (w0) [left=5pt] {$w_0$};
\foreach \n in {w1,w2}
  \node at (\n)[gray,circle,fill,inner sep=2pt]{};
\node at (w0)[gray,circle,fill,inner sep=2pt]{};
%
%\draw[blue] (xi1)--(xir)|-(b);
%\node at (-4.3,1.5) {$\ldots$};
%\node[above] at (-4.3,1.5) {$k_0$};
%\draw[blue] (-3.9,-3)--(-3.9,-3)|-(b);
%
%\draw[blue] (-3.8,-3)--(-1.1,0)--(-1.1,0)|-(b);
%\node at (-0.7,1.5) {$\ldots$};
%\node[above] at (-0.7,1.5) {$k_1$};
%\draw[blue] (-3,-3)--(-0.3,0)--(-0.3,0)|-(b);
%
%\draw[blue] (-2.5,-3)--(2.8,0)--(2.8,0)|-(b);
%\node at (3.2,1.5) {$\ldots$};
%\node[above] at (3.2,1.5) {$k_{n-1}$};
%\draw[blue] (xi1)--(3.6,0)--(3.6,0)|-(b);

%\node[gray,circle,fill,inner sep=0.8pt] (xij1) at (-3.8,-3) {};
%\node[gray,circle,fill,inner sep=0.8pt] (xij2) at (-2.5,-3) {};
%\node[gray,circle,fill,inner sep=0.8pt] (xij3) at (-3,-3) {};
%\node[below,gray] at (xij1) {$\xi_j$};
%\node[below,gray] at (xij2) {$\xi'_j$};

%\draw[dashed] (w0) to node[midway,above] (k0) {\tiny $k_0$} (w1);
%\draw[dashed] (w1) to node[midway,above] (k1) {\tiny $k_1$} (w2);
%\draw[dashed] (wn1) to node[pos=0.65,above] (kn1) {\tiny $k_{n-1}$} (wn);

\coordinate (c) at (-1.5,-1.5);

\draw[double,red] (k0)--(k0|-c);
\draw[red] (k0|-c) -- (xir);
\draw[red] (k0|-c) -- (xirk0);
%\draw[red] (k0|-c) -- (-3.9,-3);

\draw[double,red] (k1)--(k1|-c);
\draw[red] (k1|-c) -- (xi1);
\draw[red] (k1|-c) -- (xik1);
%\draw[red] (-2.2,-1.8) -- (-3,-3);

%\draw[double,red] (3,0)--(0,-1.8);
%\draw[red] (0,-1.8) -- (xij2);
%\draw[red] (0,-1.8) -- (xi1);

%\draw[gray] (-5,3) -- (-5,1);
\draw[gray] (-5,3) -- (-5,-3);
\draw[gray] (-5,-3) -- (5,-3);% node[right] {$\partial D_n$};
\draw[gray] (5,3) -- (5,-3);
\draw[gray] (5,3) -- (-5,3);
\end{tikzpicture} }}%\right).
\right)= 
\vcenter{\hbox{
\begin{tikzpicture}[scale=0.5,decoration={
    markings,
    mark=at position 0.5 with {\arrow{<}}}
    ] 
\node (w0) at (-5,1) {};
\node (w00) at (5,1) {};
\node (w1) at (2,1) {};
\node (w2) at (-2,1) {};
%\node[gray] at (1.3,0.0) {\ldots};
%\node (wn) at (4,1) {};
%\node (wn1) at (2,1) {};
\coordinate (step1) at (3.5,1);
\coordinate (step2) at (-3.5,1);
\coordinate (rouge1) at (-3,1);
\coordinate (rouge2) at (3,1);
\coordinate (rouge3) at (-4,1);
\coordinate (rouge4) at (4,1);

\node[gray,circle,fill,inner sep=0.8pt] (xir) at (-4.8,-3) {};
\node[below,gray] at (xir) {$\xi_r$};
%\node[below=5pt,gray] at (-3.2,-3) {$\ldots$};
\node[gray,circle,fill,inner sep=0.8pt] (xi1) at (-2.3,-3) {};
\node[below,gray] at (xi1) {$\xi_1$};
\node[gray,circle,fill,inner sep=0.8pt] (xirk0) at (-3.9,-3) {};
%\node[below,gray] at (xirk0) {$\xi_{r-k_0}$};
\node[gray,circle,fill,inner sep=0.8pt] (xik1) at (-3.6,-3) {};
\node[below,gray] at (xik1) {$\xi_{k_1}$};

\coordinate (a) at (-3,-3);
\coordinate (b) at (4,4);

\draw[dashed] (w0) to[bend right=90] node[pos=0.3,above] (k0) {$k_0$} (step1) ;
\draw[dashed] (step1) to[bend right=90] (step2) ;
\draw[dashed] (step2) to[bend right=90] (w1) ;
\draw[dashed] (w1) to node[pos=0.7,above] (k1) {$k_1$} (w2) ;

%\node[gray] at (w00)[above] {$w_{\infty}$};
%\node[gray] at (wn)[above] {$w_n$};
%\node[gray] at (wn1)[above] {$w_{n-1}$};
\node[gray] at (w2)[above] {$w_2$};
\node[gray] at (w1)[above] {$w_1$};
\node[gray] at (w0) [left=5pt] {$w_0$};
\foreach \n in {w1,w2}
  \node at (\n)[gray,circle,fill,inner sep=2pt]{};
\node at (w0)[gray,circle,fill,inner sep=2pt]{};
%
%\draw[blue] (xi1)--(xir)|-(b);
%\node at (-4.3,1.5) {$\ldots$};
%\node[above] at (-4.3,1.5) {$k_0$};
%\draw[blue] (-3.9,-3)--(-3.9,-3)|-(b);
%
%\draw[blue] (-3.8,-3)--(-1.1,0)--(-1.1,0)|-(b);
%\node at (-0.7,1.5) {$\ldots$};
%\node[above] at (-0.7,1.5) {$k_1$};
%\draw[blue] (-3,-3)--(-0.3,0)--(-0.3,0)|-(b);
%
%\draw[blue] (-2.5,-3)--(2.8,0)--(2.8,0)|-(b);
%\node at (3.2,1.5) {$\ldots$};
%\node[above] at (3.2,1.5) {$k_{n-1}$};
%\draw[blue] (xi1)--(3.6,0)--(3.6,0)|-(b);

%\node[gray,circle,fill,inner sep=0.8pt] (xij1) at (-3.8,-3) {};
%\node[gray,circle,fill,inner sep=0.8pt] (xij2) at (-2.5,-3) {};
%\node[gray,circle,fill,inner sep=0.8pt] (xij3) at (-3,-3) {};
%\node[below,gray] at (xij1) {$\xi_j$};
%\node[below,gray] at (xij2) {$\xi'_j$};

%\draw[dashed] (w0) to node[midway,above] (k0) {\tiny $k_0$} (w1);
%\draw[dashed] (w1) to node[midway,above] (k1) {\tiny $k_1$} (w2);
%\draw[dashed] (wn1) to node[pos=0.65,above] (kn1) {\tiny $k_{n-1}$} (wn);

\coordinate (c) at (-2.5,-2.5);

\draw[double,red] (k0) -- (k0|-c);
\draw[red] (k0|-c) -- (xir);
\draw[red] (k0|-c) -- (xirk0);
%\draw[red] (k0|-c) -- (-3.9,-3);

\draw[double,red] (k1)  to [bend left=90] (rouge1);
\draw[double,red] (rouge1)  to [bend left=90] (rouge2);
\draw[double,red] (rouge2)  to [bend left=90] (rouge3);
\draw[double,red] (rouge3)  to [bend left=90] (rouge4);
\draw[double,red] (rouge4)  to [bend left=35] (k1|-c);

%--
\draw[red] (k1|-c) -- (xi1);
\draw[red] (k1|-c) -- (xik1);
%\draw[red] (-2.2,-1.8) -- (-3,-3);

%\draw[double,red] (3,0)--(0,-1.8);
%\draw[red] (0,-1.8) -- (xij2);
%\draw[red] (0,-1.8) -- (xi1);

%\draw[gray] (-5,4) -- (-5,1);
\draw[gray] (-5,4) -- (-5,-3);
\draw[gray] (-5,-3) -- (5,-3);% node[right] {$\partial D_n$};
\draw[gray] (5,4) -- (5,-3);
\draw[gray] (5,4) -- (-5,4);
\end{tikzpicture} }}
\end{equation*}

From \cite[Remark~4.1]{Jules1}, one can find the precise rule for changing red handles from a diagram, that we use now. From a diagram, if we only modify the red-handle (leaving dashed arcs fixed), the modification is encoded by a braid. Then the diagrams correspond to the same class in $\Hrelm_r$ up to a coefficient. The coefficient showing up is the image of the corresponding braid by $\rho_r$. Knowing this, we have:
\begin{equation*}
\vcenter{\hbox{
\begin{tikzpicture}[scale=0.5,decoration={
    markings,
    mark=at position 0.5 with {\arrow{<}}}
    ] 
\node (w0) at (-5,1) {};
\node (w00) at (5,1) {};
\node (w1) at (2,1) {};
\node (w2) at (-2,1) {};
%\node[gray] at (1.3,0.0) {\ldots};
%\node (wn) at (4,1) {};
%\node (wn1) at (2,1) {};
\coordinate (step1) at (3.5,1);
\coordinate (step2) at (-3.5,1);
\coordinate (rouge1) at (-3,1);
\coordinate (rouge2) at (3,1);
\coordinate (rouge3) at (-4,1);
\coordinate (rouge4) at (4,1);

\node[gray,circle,fill,inner sep=0.8pt] (xir) at (-4.8,-3) {};
\node[below,gray] at (xir) {$\xi_r$};
%\node[below=5pt,gray] at (-3.2,-3) {$\ldots$};
\node[gray,circle,fill,inner sep=0.8pt] (xi1) at (-2.3,-3) {};
\node[below,gray] at (xi1) {$\xi_1$};
\node[gray,circle,fill,inner sep=0.8pt] (xirk0) at (-3.9,-3) {};
%\node[below,gray] at (xirk0) {$\xi_{r-k_0}$};
\node[gray,circle,fill,inner sep=0.8pt] (xik1) at (-3.6,-3) {};
\node[below,gray] at (xik1) {$\xi_{k_1}$};

\coordinate (a) at (-3,-3);
\coordinate (b) at (4,4);

\draw[dashed] (w0) to[bend right=90] node[pos=0.3,above] (k0) {$k_0$} (step1) ;
\draw[dashed] (step1) to[bend right=90] (step2) ;
\draw[dashed] (step2) to[bend right=90] (w1) ;
\draw[dashed] (w1) to node[pos=0.7,above] (k1) {$k_1$} (w2) ;

%\node[gray] at (w00)[above] {$w_{\infty}$};
%\node[gray] at (wn)[above] {$w_n$};
%\node[gray] at (wn1)[above] {$w_{n-1}$};
\node[gray] at (w2)[above] {$w_2$};
\node[gray] at (w1)[above] {$w_1$};
\node[gray] at (w0) [left=5pt] {$w_0$};
\foreach \n in {w1,w2}
  \node at (\n)[gray,circle,fill,inner sep=2pt]{};
\node at (w0)[gray,circle,fill,inner sep=2pt]{};
%
%\draw[blue] (xi1)--(xir)|-(b);
%\node at (-4.3,1.5) {$\ldots$};
%\node[above] at (-4.3,1.5) {$k_0$};
%\draw[blue] (-3.9,-3)--(-3.9,-3)|-(b);
%
%\draw[blue] (-3.8,-3)--(-1.1,0)--(-1.1,0)|-(b);
%\node at (-0.7,1.5) {$\ldots$};
%\node[above] at (-0.7,1.5) {$k_1$};
%\draw[blue] (-3,-3)--(-0.3,0)--(-0.3,0)|-(b);
%
%\draw[blue] (-2.5,-3)--(2.8,0)--(2.8,0)|-(b);
%\node at (3.2,1.5) {$\ldots$};
%\node[above] at (3.2,1.5) {$k_{n-1}$};
%\draw[blue] (xi1)--(3.6,0)--(3.6,0)|-(b);

%\node[gray,circle,fill,inner sep=0.8pt] (xij1) at (-3.8,-3) {};
%\node[gray,circle,fill,inner sep=0.8pt] (xij2) at (-2.5,-3) {};
%\node[gray,circle,fill,inner sep=0.8pt] (xij3) at (-3,-3) {};
%\node[below,gray] at (xij1) {$\xi_j$};
%\node[below,gray] at (xij2) {$\xi'_j$};

%\draw[dashed] (w0) to node[midway,above] (k0) {\tiny $k_0$} (w1);
%\draw[dashed] (w1) to node[midway,above] (k1) {\tiny $k_1$} (w2);
%\draw[dashed] (wn1) to node[pos=0.65,above] (kn1) {\tiny $k_{n-1}$} (wn);

\coordinate (c) at (-2.5,-2.5);

\draw[double,red] (k0) -- (k0|-c);
\draw[red] (k0|-c) -- (xir);
\draw[red] (k0|-c) -- (xirk0);
%\draw[red] (k0|-c) -- (-3.9,-3);

\draw[double,red] (k1)  to [bend left=90] (rouge1);
\draw[double,red] (rouge1)  to [bend left=90] (rouge2);
\draw[double,red] (rouge2)  to [bend left=90] (rouge3);
\draw[double,red] (rouge3)  to [bend left=90] (rouge4);
\draw[double,red] (rouge4)  to [bend left=35] (k1|-c);

%--
\draw[red] (k1|-c) -- (xi1);
\draw[red] (k1|-c) -- (xik1);
%\draw[red] (-2.2,-1.8) -- (-3,-3);

%\draw[double,red] (3,0)--(0,-1.8);
%\draw[red] (0,-1.8) -- (xij2);
%\draw[red] (0,-1.8) -- (xi1);

%\draw[gray] (-5,4) -- (-5,1);
\draw[gray] (-5,4) -- (-5,-3);
\draw[gray] (-5,-3) -- (5,-3);% node[right] {$\partial D_n$};
\draw[gray] (5,4) -- (5,-3);
\draw[gray] (5,4) -- (-5,4);
\end{tikzpicture} }} = (-1)^{k_1} (q^{2l})^{3} (-t)^{-3\frac{k_1(k_1-1)}{2}} \left(
\vcenter{\hbox{
\begin{tikzpicture}[scale=0.5,decoration={
    markings,
    mark=at position 0.5 with {\arrow{<}}}
    ] 
\node (w0) at (-5,1) {};
\node (w00) at (5,1) {};
\node (w1) at (2,1) {};
\node (w2) at (-2,1) {};
%\node[gray] at (1.3,0.0) {\ldots};
%\node (wn) at (4,1) {};
%\node (wn1) at (2,1) {};
\coordinate (step1) at (3.5,1);
\coordinate (step2) at (-3.5,1);
\coordinate (rouge1) at (-3,1);
\coordinate (rouge2) at (3,1);
\coordinate (rouge3) at (-4,1);
\coordinate (rouge4) at (4,1);

\node[gray,circle,fill,inner sep=0.8pt] (xir) at (-4.8,-3) {};
\node[below,gray] at (xir) {$\xi_r$};
%\node[below=5pt,gray] at (-3.2,-3) {$\ldots$};
\node[gray,circle,fill,inner sep=0.8pt] (xi1) at (-2.3,-3) {};
\node[below,gray] at (xi1) {$\xi_1$};
\node[gray,circle,fill,inner sep=0.8pt] (xirk0) at (-3.9,-3) {};
%\node[below,gray] at (xirk0) {$\xi_{r-k_0}$};
\node[gray,circle,fill,inner sep=0.8pt] (xik1) at (-3.6,-3) {};
\node[below,gray] at (xik1) {$\xi_{k_1}$};

\coordinate (a) at (-3,-3);
\coordinate (b) at (4,4);

\draw[dashed] (w0) to[bend right=90] node[pos=0.3,above] (k0) {$k_0$} (step1) ;
\draw[dashed] (step1) to[bend right=90] (step2) ;
\draw[dashed] (step2) to[bend right=90] (w1) ;
\draw[dashed] (w1) to node[pos=0.7,above] (k1) {$k_1$} (w2) ;

%\node[gray] at (w00)[above] {$w_{\infty}$};
%\node[gray] at (wn)[above] {$w_n$};
%\node[gray] at (wn1)[above] {$w_{n-1}$};
\node[gray] at (w2)[above] {$w_2$};
\node[gray] at (w1)[above] {$w_1$};
\node[gray] at (w0) [left=5pt] {$w_0$};
\foreach \n in {w1,w2}
  \node at (\n)[gray,circle,fill,inner sep=2pt]{};
\node at (w0)[gray,circle,fill,inner sep=2pt]{};
%
%\draw[blue] (xi1)--(xir)|-(b);
%\node at (-4.3,1.5) {$\ldots$};
%\node[above] at (-4.3,1.5) {$k_0$};
%\draw[blue] (-3.9,-3)--(-3.9,-3)|-(b);
%
%\draw[blue] (-3.8,-3)--(-1.1,0)--(-1.1,0)|-(b);
%\node at (-0.7,1.5) {$\ldots$};
%\node[above] at (-0.7,1.5) {$k_1$};
%\draw[blue] (-3,-3)--(-0.3,0)--(-0.3,0)|-(b);
%
%\draw[blue] (-2.5,-3)--(2.8,0)--(2.8,0)|-(b);
%\node at (3.2,1.5) {$\ldots$};
%\node[above] at (3.2,1.5) {$k_{n-1}$};
%\draw[blue] (xi1)--(3.6,0)--(3.6,0)|-(b);

%\node[gray,circle,fill,inner sep=0.8pt] (xij1) at (-3.8,-3) {};
%\node[gray,circle,fill,inner sep=0.8pt] (xij2) at (-2.5,-3) {};
%\node[gray,circle,fill,inner sep=0.8pt] (xij3) at (-3,-3) {};
%\node[below,gray] at (xij1) {$\xi_j$};
%\node[below,gray] at (xij2) {$\xi'_j$};

%\draw[dashed] (w0) to node[midway,above] (k0) {\tiny $k_0$} (w1);
%\draw[dashed] (w1) to node[midway,above] (k1) {\tiny $k_1$} (w2);
%\draw[dashed] (wn1) to node[pos=0.65,above] (kn1) {\tiny $k_{n-1}$} (wn);

\coordinate (c) at (-2.5,-2.5);

\draw[double,red] (k0) -- (k0|-c);
\draw[red] (k0|-c) -- (xir);
\draw[red] (k0|-c) -- (xirk0);
%\draw[red] (k0|-c) -- (-3.9,-3);

\draw[double,red] (k1)  -- (k1|-c);
%\draw[double,red] (rouge1)  to [bend left=90] (rouge2);
%\draw[double,red] (rouge2)  to [bend left=90] (rouge3);
%\draw[double,red] (rouge3)  to [bend left=90] (rouge4);
%\draw[double,red] (rouge4)  to [bend left=35] (k1|-c);

%--
\draw[red] (k1|-c) -- (xi1);
\draw[red] (k1|-c) -- (xik1);
%\draw[red] (-2.2,-1.8) -- (-3,-3);

%\draw[double,red] (3,0)--(0,-1.8);
%\draw[red] (0,-1.8) -- (xij2);
%\draw[red] (0,-1.8) -- (xi1);

%\draw[gray] (-5,4) -- (-5,1);
\draw[gray] (-5,4) -- (-5,-3);
\draw[gray] (-5,-3) -- (5,-3);% node[right] {$\partial D_n$};
\draw[gray] (5,4) -- (5,-3);
\draw[gray] (5,4) -- (-5,4);
\end{tikzpicture} }} \right). 
\end{equation*}
See \cite[Example~4.12]{Jules2} for an inspiring example, from which the above equality can be seen as a three time iteration. The factor $(q^{2l})^{-3 k_1}$ comes from the fact that the red handle has winding number $-3$ around punctures, the $t^{-3\frac{k_1(k_1-1)}{2}}$ is due to the framing of the red handle (the $k_1$ points in the handle are given $3$ full twists while winding around punctures), the $(-1)^{-3\frac{k_1(k_1-1)}{2}}$ appears because of latter full twists permuting coordinates in the parametrization of the simplex, and finally $(-1)^{k_1}$ comes from the fact that the orientation of the interval over which configurations of the simplex are supported is reversed. In \cite{Jules1} (Remark~4.1 for example) both terms coming from the number of full twists given by the red handle are reunited under the same variable denoted $\mt:=-t$. We denote $U'(k_0,k_1)$ the diagram involved in the right term of the above equality which naturally corresponds to a class in $\Hrelm_r$. Now we recall that to compute the Jones polynomials of the trefoil knot out of the formula from Theorem \ref{pairingformulaforJones}, we must compute the pairings:
\begin{align*}
\langle \sigma_1^3 \left( U(k_0,k_1) \right) , B(k_0,k_1) \rangle_{\Laurent} & = (-1)^{k_1} (q^{2l})^{-3 k_1} (-1)^{-3\frac{k_1(k_1-1)}{2}}  t^{-3\frac{k_1(k_1-1)}{2}} \langle U'(k_0,k_1) , B(k_0,k_1) \rangle_{\Laurent} \\
%& = (-1)^{k_1} (q^{2l})^{-3 k_1} (-1)^{-3\frac{k_1(k_1-1)}{2}}  t^{-3\frac{k_1(k_1-1)}{2}} \CP_{k_0,k_1},
\end{align*}
where the diagram illustrating the computation of $\langle U'(k_0,k_1) , B(k_0,k_1) \rangle_{\Laurent}$ is the following one:
\begin{equation}\label{intersectiontrivial}
\vcenter{\hbox{
\begin{tikzpicture}[scale=0.5,decoration={
    markings,
    mark=at position 0.5 with {\arrow{<}}}
    ] 
\node (w0) at (-5,1) {};
\node (w00) at (5,1) {};
\node (w1) at (2,1) {};
\node (w2) at (-2,1) {};
%\node[gray] at (1.3,0.0) {\ldots};
%\node (wn) at (4,1) {};
%\node (wn1) at (2,1) {};
\coordinate (step1) at (3.5,1);
\coordinate (step2) at (-3.5,1);
\coordinate (rouge1) at (-3,1);
\coordinate (rouge2) at (3,1);
\coordinate (rouge3) at (-4,1);
\coordinate (rouge4) at (4,1);

\node[gray,circle,fill,inner sep=0.8pt] (xir) at (-4.8,-3) {};
\node[below,gray] at (xir) {$\xi_r$};
%\node[below=5pt,gray] at (-3.2,-3) {$\ldots$};
\node[gray,circle,fill,inner sep=0.8pt] (xi1) at (-2.3,-3) {};
\node[below,gray] at (xi1) {$\xi_1$};
\node[gray,circle,fill,inner sep=0.8pt] (xirk0) at (-3.7,-3) {};
%\node[below,gray] at (xirk0) {$\xi_{r-k_0}$};
\node[gray,circle,fill,inner sep=0.8pt] (xik1) at (-3.3,-3) {};
\node[below,gray] at (xik1) {$\xi_{k_1}$};

\coordinate (a) at (-3,-3);
\coordinate (b) at (4,4);

\draw[dashed] (w0) to[bend right=90] node[pos=0.35,above] (k0) {$k_0$} (step1) ;
\draw[dashed] (step1) to[bend right=90] (step2) ;
\draw[dashed] (step2) to[bend right=90] (w1) ;
\draw[dashed] (w1) to node[pos=0.5,above] (k1) {$k_1$} (w2) ;

%\node[gray] at (w00)[above] {$w_{\infty}$};
%\node[gray] at (wn)[above] {$w_n$};
%\node[gray] at (wn1)[above] {$w_{n-1}$};
\node[gray] at (w2)[above] {$w_2$};
\node[gray] at (w1)[above] {$w_1$};
\node[gray] at (w0) [left=5pt] {$w_0$};
\foreach \n in {w1,w2}
  \node at (\n)[gray,circle,fill,inner sep=2pt]{};
\node at (w0)[gray,circle,fill,inner sep=2pt]{};
\draw[blue] (xir)--(xir|-b);
\draw[blue] (xirk0)--(xirk0|-b);
\node[blue] at (-4.3,1.5) {$\ldots$};
\node[blue,above] at (-4.3,1.5) {$k_0$};
%\draw[blue] (-3.9,-3)--(-3.9,-3)|-(b);
%
\draw[blue] (xik1)--(-1.3,0|-c)--(-1.3,0|-b);
\draw[blue] (xi1)--(1.3,0|-c)--(1.3,0|-b);
\node[blue] at (0,3.5) {$\ldots$};
\node[above,blue] at (0,3.5) {$k_1$};
%\draw[blue] (-3,-3)--(-0.3,0)--(-0.3,0)|-(b);
%
%\draw[blue] (-2.5,-3)--(2.8,0)--(2.8,0)|-(b);
%\node at (3.2,1.5) {$\ldots$};
%\node[above] at (3.2,1.5) {$k_{n-1}$};
%\draw[blue] (xi1)--(3.6,0)--(3.6,0)|-(b);

%\node[gray,circle,fill,inner sep=0.8pt] (xij1) at (-3.8,-3) {};
%\node[gray,circle,fill,inner sep=0.8pt] (xij2) at (-2.5,-3) {};
%\node[gray,circle,fill,inner sep=0.8pt] (xij3) at (-3,-3) {};
%\node[below,gray] at (xij1) {$\xi_j$};
%\node[below,gray] at (xij2) {$\xi'_j$};

%\draw[dashed] (w0) to node[midway,above] (k0) {\tiny $k_0$} (w1);
%\draw[dashed] (w1) to node[midway,above] (k1) {\tiny $k_1$} (w2);
%\draw[dashed] (wn1) to node[pos=0.65,above] (kn1) {\tiny $k_{n-1}$} (wn);

\coordinate (c) at (-2.5,-2.5);

\draw[double,red] (k0) -- (k0|-c);
\draw[red] (k0|-c) -- (xir);
\draw[red] (k0|-c) -- (xirk0);
%\draw[red] (k0|-c) -- (-3.9,-3);

\draw[double,red] (k1)  -- (k1|-c);
%\draw[double,red] (rouge1)  to [bend left=90] (rouge2);
%\draw[double,red] (rouge2)  to [bend left=90] (rouge3);
%\draw[double,red] (rouge3)  to [bend left=90] (rouge4);
%\draw[double,red] (rouge4)  to [bend left=35] (k1|-c);

%--
\draw[red] (k1|-c) -- (xi1);
\draw[red] (k1|-c) -- (xik1);
%\draw[red] (-2.2,-1.8) -- (-3,-3);

%\draw[double,red] (3,0)--(0,-1.8);
%\draw[red] (0,-1.8) -- (xij2);
%\draw[red] (0,-1.8) -- (xi1);

%\draw[gray] (-5,4) -- (-5,1);
\draw[gray] (-5,4) -- (-5,-3);
\draw[gray] (-5,-3) -- (5,-3);% node[right] {$\partial D_n$};
\draw[gray] (5,4) -- (5,-3);
\draw[gray] (5,4) -- (-5,4);
\end{tikzpicture} }} %\right). 
\end{equation}
The picture is now endowed with blue arcs corresponding to the barcode, consisting of $k_0$ blue arcs passing between $w_0$ and $w_1$ and $k_1$ arcs between $w_1$ and $w_2$. To find an intersection point between classes corresponding to $U'(k_0,k_1)$ and $B(k_0,k_1)$ we must consider a configuration with one and only one point on each blue arc. Thus there is only one intersection configuration denoted ${\pmb p} = (p_1 , \ldots, p_r)$ for which $p_1 , \ldots , p_{k_0}$ are the intersection point between the first $k_0$ blue arc and the dashed arc indexed by $k_0$. It is unique as we must choose $k_1$ points on the indexed $k_1$ dashed arcs for $(p_{k_0}, \ldots ,p_{k_r})$, with only one choice for intersecting blue arcs. 

One can check that the loop of $X_r$ starting at $\lbrace \xi_1 , \ldots , \xi_r \rbrace$, going to $U'(k_0 , k_1)$ following red handles, then to ${\pmb p}$ along $U'(k_0,k_1)$, finally back to the base point following blue arcs from $B(k_0,k_1)$ is trivial. So that for all $(k_0,k_1) \in \BN^2$, we have:
\[
\langle U'(k_0,k_1) , B(k_0,k_1) \rangle_{\Laurent} = 1.
\]
Finally, we have:
\[
\langle \sigma_1^3 \left( U(k_0,k_1) \right) , B(k_0,k_1) \rangle_{\Laurent} = (-1)^{k_1} (q^{2l})^{-3 k_1}  (-1)^{-3\frac{k_1(k_1-1)}{2}} t^{-3\frac{k_1(k_1-1)}{2}} . 
\]
For the sake of computing the colored Jones polynomials, we are interested in $t=-q^{-2}$ (see \cite[Theorem~1]{Jules1}). Then, if $K$ designates the trefoil knot and $l \in \BN$:
\begin{align*}
\Jones_K(l+1)  & =   q^{-3-2l}\left( 1+ \sum_{r=1}^{2l} \sum_{\begin{array}{c} (k_0,\ldots,k_{n-1})/ \\ \sum k_i=r \end{array}} \left\langle \sigma_1^3 \left( U(k_0,\ldots, k_{n-1}) \right) , B(k_0,\ldots, k_{n-1}) \rangle_{\aug^l}\right. q^{2r} \right) \\
&  =   q^{-3-2l}\left( 1 + \sum_{r=0}^{2l} \sum_{\begin{array}{c} (k_0,\ldots,k_{n-1})/ \\ \sum k_i=r \end{array}} (-1)^{k_1} (q^{2l})^{-3 k_1}  (q^{-2})^{-3\frac{k_1(k_1-1)}{2}} q^{2r} \right) .
\end{align*}
The usual Jones polynomial corresponds to the case $l=2$. In that case one has:
\[
\Jones_K(2)= q^{-3}\left( q^{-2} + (1-q^{-6})+q^{2} \right) =  q^{-1} + q^{-3} + q^{-5} - q^{-9}. 
\]

\begin{rmk}[Conventions]
The Jones polynomial of the trefoil computed above is the one computed in \cite[Figure~4]{BN} (see the target of the arrow in the example) up to $q\mapsto q^{-1}$ (we compute here the mirror image, it depends on the choice for the braiding of $\Uq$). One sees:
\[
q^{-1} + q^{-3} + q^{-5} - q^{-9} = (q+q^{-1}) (q^{-2} +q^{-6}-q^{-8}) .
\]
In the second writing of the polynomial, one notices that the Jones polynomial factorizes by $(q+q^{-1})$ corresponding to the Jones polynomial of the unknot. The other factor is the {\em normalized} Jones polynomial. 
\end{rmk}

\newpage
\thispagestyle{empty}

\end{document}